\documentclass[12pt]{amsart}

\usepackage{latexsym}
\usepackage{amsfonts}
\usepackage{amsmath}
\usepackage{amsthm}
\usepackage{amssymb}
\usepackage{enumerate}
\usepackage{color}
\usepackage{float}
\usepackage{tikz-cd}
\usepackage{mathrsfs}
\usepackage{graphicx}
\usepackage{array}
\usepackage{adjustbox}
\usepackage{extarrows}
\usepackage{enumerate}
\usepackage{caption}

\theoremstyle{plain}
\newtheorem{theorem}{Theorem}
\newtheorem{prop}[theorem]{Proposition}
\newtheorem{lemma}[theorem]{Lemma}

\newtheorem{rem}[theorem]{Remark}

\theoremstyle{plain}

\theoremstyle{definition}
\newtheorem{definition}[theorem]{Definition}

\newcommand\R{\mathbb{R}}

\newcommand\C{\mathbb{C}}

\newcommand\Z{\mathbb{Z}}

\newcommand\w{\wedge}
\newcommand\cc{\check}

\newcommand\scalemath[2]{\scalebox{#1}{\mbox{\ensuremath{\displaystyle #2}}}}

\title{SYZ mirror symmetry of solvmanifolds}
\thanks{This work was supported by GNSAGA of INdAM}
\address{Dipartimento di Ingegneria e Scienze dell'Informazione e Matematica \\ Universit\`a dell'Aquila\\
	via Vetoio\\ 67100 L'Aquila\\ Italy}
\email{lucio.bedulli@univaq.it}
\email{alessandro.vannini@graduate.univaq.it}

\author{Lucio Bedulli and Alessandro Vannini}
\begin{document}
%\title{SYZ mirror symmetry of solvmanifolds}
\maketitle
%\author{Lucio Bedulli and Alessandro Vannini}

%\abstract{We explicitly find SYZ mirror symmetric partners of all known compact 6-dimensional symplectic half-flat completely solvable solvmanifolds that admit a semi-flat structure. We exhibit the behavior of the relevant differential forms under the Fourier-Mukai transform described in \cite{LTY}.}
\begin{abstract}
	We present an effective construction of non-K\"ahler supersymmetric mirror pairs in the sense of \cite{LTY} starting from left-invariant affine structures on Lie groups.
	Applying this construction we explicitly find SYZ mirror symmetric partners of all known compact 6-dimensional completely solvable solvmanifolds that admit a semi-flat type IIA structure.
%	 We exhibit the behavior of the relevant differential forms under the Fourier-Mukai transform described in \cite{LTY}.
\end{abstract}

\section{Introduction}

%Mirror symmetry of Calabi-Yau manifolds has been a fascinating subject since its appearance in Physics literature.\\
%...\\
The Strominger-Yau-Zaslow (SYZ) conjecture (see \cite{SYZ}) tries to describe mirror symmetry of Calabi-Yau manifolds in terms of dual Lagrangian torus fibrations.

In this paper we deal with a non-K\"ahler version of SYZ mirror symmetry where the correspondence between symplectic and complex structures of the partners is made explicit through Fourier-Mukai transform. The role of the Dolbeault cohomology is replaced by the Bott-Chern cohomology on the complex side and by a refined version of the Tseng-Yau cohomology on the symplectic side. The procedure is thouroghly explained in \cite{LTY}.
Here we will mainly stick to the case of manifolds of real dimension 6, which is the ambient where originally mirror symmetry made its appearance.

%An important feature of this version is 
One of the aims of the present paper is to show that, unlike in the K\"ahler case, it is possible to find many interesting mirror pairs of compact 6-manifolds without the need of singularities in the fibrations.
The first and only example of this kind known so far is the nilmanifold featured in \cite{LTY}.

Nilmanifolds and more generally solvmanifolds are of course a natural ambient to look at in the seek of such structures. For example there are plenty of explicit symplectic structures on non-abelian nilmanifolds and none of these can be K\"ahler.
%({\color{red} come funziona il caso solvable?}).\\ 

Moreover all 6-dimensional solvable Lie algebras admitting symplectic half-flat structures are classified (see \cite{FMOU}).

Compact quotients of the corresponding simply connected Lie groups are the best known explicit examples of compact manifolds carrying a type IIA structure.

While reinterpreting the example given in \cite{LTY}, the aim of the present paper is to explicitly find SYZ mirror symmetric partners (in the sense of \cite{LTY}) of all known compact 6-dimensional symplectic half-flat completely solvable solvmanifolds that admit a semi-flat structure.

The starting observation is that all such examples are indeed quotients of Lie groups having a particular structure of semi-direct product.
This semi-direct product structure is intimately related to a left invariant affine structure on 3-dimensional solvable Lie groups, hence to a Lagrangian torus bundle on suitable quotients of it.

It is this semi-direct product structure that allows us to explicitly find the non-singular (i.e. semi-flat) dual torus fibrations, hence the mirror partner.
This is explained in section \ref{construction}. \\

The main result is summarized in the following

\begin{theorem}
	
Let $(X,\omega,\Omega)$ be a compact solvmanifold endowed with a semi-flat left-invariant IIA structure.
Then its SYZ mirror partner $(\cc X,\cc \omega,\cc \Omega)$ is a solvmanifold  endowed with a semi-flat left-invariant IIB structure.\\
In table \ref{mirrortable} all  the corresponding pairs of Lie algebras as well as the Bott-Chern and refined Tseng-Yau numbers are listed.
%Let $(\frak g,\omega,\Omega)$ be a type IIA solvable Lie algebra appearing in table \ref{algebreIIA}. If it is completely solvable and semi-flat then its mirror type IIB solvable Lie algebra is listed in table \ref{mirrortable}.\\
%Moreover in all these cases the corresponding Lie groups admit a lattice giving rise to a compact semi-flat supersymmetric mirror pair.
\end{theorem}

In table \ref{mirrortable} we also write down the Tseng-Yau and Bott-Chern numbers of the algebras involved that realize the mirror symmetric {\em non-K\"ahler Hodge diamonds}.

As an upshot of our constructions we find a compact type IIA manifold $X$ admitting two inequivalent Lagrangian torus fibrations giving rise to two non-isomorphic semi-flat mirror pairs: the complex IIB partners $\cc X$ and $\cc X'$ are not even diffeomorphic (see \ref{twistedR3IIB} and \ref{untwistedH3IIB}).

As a by-product we also find a new balanced metric on a nilmanifold which is missing from the classification given in \cite{LUV}, see subsection \ref{twistedH3IIB}, Remark \ref{remUgarte}.

%Finally in section \ref{FMsec} we exhibit the behavior of the relevant differential forms under the Fourier-Mukai transform described in \cite{LTY}.\\
A systematic study of all the semi-flat ${\rm SU}(3)$-mirror pairs coming from left-invariant affine structures using the classification of \cite{FG} is carried over in a forthcoming paper.\\

Several related results should be mentioned.
Firstly in \cite{Nguiffo} invariant symplectic structures on $T^*G$ are constructed on a Lie group $G$ carrying an invariant affine structure. The analogous construction of invariant complex structures for $TG$ can be found in \cite{BarberisDotti}.

%Moreover in \cite{CLP} the authors showed that  Lie algebras 
%which are constructed via dual semi-direct product have isomorphic differential Gerstenhaber algebras realizing a sort of algebraic {\em weak} mirror symmetry. In the same work one can also find the full list of all pairs of {\em nilpotent} 6-dimensional Lie algebras constructed via dual semi-direct product.

Moreover in \cite{CLP} the authors list all the pairs of {\em nilpotent} 6-dimensional Lie algebras constructed via dual semi-direct product and show that they have isomorphic differential Gerstenhaber algebras realizing a sort of algebraic {\em weak} mirror symmetry. 

%{\color{red} It is to be noted that our approach is different since it is topological in nature and deals from the beginning with compact manifolds and torus fibrations over them.}
%
%
%{\color{blue} It is remarkable that we obtained the same construction starting from ....}
%
%{\color{red} Our approach is different since it is topological in nature and deals from the beginning with compact manifolds and torus fibrations over them.}

\bigskip

{\bf Acknowledgments}. The authors would like to thank L. Ugarte and A. Raffero for some comments and remarks that helped to improve the presentation. Moreover the second named author would like to thank D. Angella and M. Garcia-Fernandez for useful discussions and the interest shown in this work.  The authors would also like to thank M. L. Barberis and J. Lauret for pointing out appropriate references.
	
\section{Preliminaries} 
\subsection{Affine structures and dual torus bundles}
\label{affine}
%\section{Integral affine structures and (Lagrangian) torus bundles}
For details about affine structures see \cite{Goldmanbook, Thurston} , here we will just recall the notions relevant to our construction.

An affine structure on a $n$-manifold is an atlas whose transition functions are restrictions of affine maps.
%If $G$ is a Lie group, an affine structure on $G$ is said to be {\em left-invariant} if left multiplications are affine transformations in affine local coordinates.\\

Any affine structure on a $n$-manifold $M$ defines a {\em developing map} $D: \widetilde M \to \R^n$, where $\widetilde M \to M$
 is the universal covering, and a holonomy representation $h: \pi_1(M) \to {\rm Aff}(\R^n) $ (for the precise definition see \cite{Thurston} or \cite{Goldmanbook}).\\
The affine structure is said to be {\em complete} if the developing map is a homeomorphism.\\
Viceversa starting from a pair $(D,h)$ where $h: \pi_1(M) \to  {\rm Aff}(\R^n)$ is a homomorphism and $D:  \widetilde M \to \R^n$ is a homeomorphism equivariant with respect to the $\pi_1(M)$ and $h(\pi_1(M))$ actions one can recover a unique complete affine structure on $M$ such that $D$ is the induced developing map and $h$ is the induced holonomy representation.

An affine structure is {\em integral} if the linear part of transition functions is integral (i.e. takes values in ${\rm GL}(n,\Z)$). 
Any integral affine structure $\mathcal A$ on a manifold $B$ defines a Lagrangian bundle $X \to B$ over $B$ in the following way. 
Let $r_1,\ldots,r_n$ be local affine coordinates on $U \subseteq B$.
Then for every $q \in U$ we can define $\Lambda^*_q$ to be the integral lattice of $T^*_qB$ generated by $dr_1,\ldots,dr_n$. This definition does not depend on the choice of the local affine chart. The manifold $T^*B/\Lambda^*$ will be an $n$-torus bundle over $B$. 
Furthermore the canonical symplectic structure of $T^*B$ passes to the quotient and 
the fibers are indeed Lagrangian with respect to it. 

Viceversa every Lagrangian torus bundle $X \to B$ induces an integral affine structure on $B$.  Over the affine coordinate charts the torus bundle is locally isomorphic to one of the form $T^*B/\Lambda^* \to B$: this is a consequence of the famous {\em Arnol'd-Liouville Theorem} in classical mechanics which in particular establishes the existence of the so-called {\em action-angle} coordinates. We remark here that the action coordinates are exactly the coordinates associated to the developing map of the integral affine structure.

Again for details about Lagrangian torus bundles we refer to the classical paper of Duistermaat \cite{Duistermaat}. See also \cite[section 3]{CastanoMatessi} for a good presentation of this topic.
% and also \cite{CastanoMatessi}.\\
%Let $r_1,\ldots,r_n$ be local affine coordinates on $U \subseteq B$.
%Then for every $q \in U$ we can define $\Lambda^*_q$ to be the integral lattice of $T^*_qM$ generated by $dr_1,\ldots,dr_n$. This definition does not depend on the choice of the local affine chart. The manifold $T^*M/\Lambda^*$ will be an $n$-torus bundle over $M$.\\
%Moreover the canonical symplectic structure on $T^*M$ descends to the quotient $T^*M/\Lambda^*$ and the fibers of the canonical projection $T^*M/\Lambda^* \to M$ are easily seen to be Lagrangian with respect to this.
%Every Lagrangian torus bundle over a manifold $M$ is locally isomorphic to one of the form $T^*M/\Lambda^* \to M$ where $\Lambda^*$ is induced by some integral affine structure on $M$.\\

Given any Lagrangian torus bundle $X \to B$ together with its action-angle coordinates $ r_1,\ldots,r_n, \theta_1,\ldots,\theta_n$ we can define the
{\em dual torus bundle} $\cc X \to B$ simply by dualizing the transition functions.  
 Locally this is isomorphic to the torus bundle $TB / \Lambda \to B$ obtained by considering the fiberwise lattice $\Lambda \subset TB$ locally generated by $\frac{\partial}{\partial r_1},\ldots,\frac{\partial}{\partial r_n}$. If we denote by $\cc\theta_k$ the fiber coordinates corresponding to the action coordinates $r_k$, we get local complex coordinates $z_k=\cc\theta_k + i r_k$ on $TB/\Lambda$ hence on $\cc X$. With respect to this complex structure the fibers of $\cc X \to B$ are totally real. 

A symplectic manifold $X$ together with a Lagrangian torus fibration and its complex dual $\cc X$ are said to form a {\em semi-flat mirror pair} in \cite{LTY}.

\subsection{Non-K\"ahler SYZ mirror symmetry}
Here we briefly describe the non K\"ahler version of SYZ mirror symmetry as presented in \cite{LTY}. 
First recall that an ${\rm SU}(n)$-structure is determined on a real $2n$-manifold $X$ by a pair of differential forms $(\omega,\Omega)$, where 
\begin{enumerate}
	\item $\Omega$ is a decomposable complex $n$-form such that setting $$T^{0,1}X=\{v \in TX\otimes C: \iota_v\Omega=0\}$$ and $T^{1,0}X=\overline{T^{0,1}X}$ we have a splitting $$TX \otimes \C = T^{1,0}X \oplus T^{0,1}X$$ inducing an almost complex structure $J$.
	\item $\omega$ is a positive $(1,1)$-form with respect to $J$ satisfying the normalization condition $$\Omega \wedge \bar\Omega=\frac{(-2i)^n}{n!}\omega^n\,.$$
\end{enumerate}  
It is easy to prove that the almost complex structure $J$ defined by $\Omega$ is integrable if and only if $\Omega$ is closed.

In the 3-dimensional case we have the following
\begin{definition} An ${\rm SU}(3)$-manifold $(X,\omega,\Omega)$ is said to be supersymmetric of type IIA if
$d\omega = 0$ and $d \, \rm{Re} \, \Omega = 0$.
\end{definition}
\begin{definition} An ${\rm SU}(3)$-manifold $(X,\omega,\Omega)$ is said to be supersymmetric of type IIB if $d(\omega^2) = 0$ and $d\Omega = 0$.
\end{definition}
Note that type IIB manifolds are {\em balanced} complex manifolds with holomorphically trivial canonical bundle while type IIA manifolds are often called symplectic {\em half-flat} manifolds. 
%The first main result of \cite{LTY} is that the Fourier-Mukai transform defines a correspondence between type IIA structures on the symplectic side of a semi-flat mirror pair and type IIB structures on the complex side.
% 
%The second main result of \cite{LTY} states that the $T$-invariant Bott-Chern cohomology of the complex side is related to  a refinement of the $T$-invariant Tseng-Yau cohomology of the symplectic side induced by the Lagrangian torus bundle structure.

%\subsection{Fourier-Mukai transform}
%Once we are provided with the datum of two dual (in the sense of the previous section) torus bundles, we are finally able to define the tool which realizes the \textit{mirror symmetry} exchange.
Let $\pi:(X,\omega)\rightarrow B$ be a Lagrangian torus bundle
 %over a 3-dimensional base $B$ 
 and let $\check\pi:(\check X,\cc\Omega)\rightarrow B$ be its dual so that $(X,\omega)$ and $(\cc X,\cc\Omega)$ form a semi-flat mirror pair.

We denote by $\mathcal{A}^k_{B}(X,\mathbb{C})$ the space of complex-valued $k$-forms on $X$ which depend only on the base, also called {\em semi-flat forms}. An element $\phi\in\mathcal{A}^k_B(X,\mathbb{C})$ is locally written as

$$\phi=\sum_{I,J}a_{IJ}(r) d\theta_I\w dr_J$$

where $I=(i_1,\dots,i_p)$, $J=(j_1,\dots,j_q)$ are multi-indices and $p+q=k$,
$(r_i,\theta_i)$ are action-angle coordinates and $a_{IJ}(r)$ are complex-valued functions on $B$. %with $r=(r_1,\dots,r_n)$.

%Let $\Delta \subset TM$ be the Lagrangian distribution coming from the Lagrangian bundle structure of $(X,\omega)\to B$. Let  $\Delta^\perp$ be its orthogonal complement. This choice allow us to decompose the space of differential forms:
%Since we want to work with $(p,q)$-forms also on the symplectic side, it is necessary to choose an integrable distribution to split the (co)tangent bundle. The most natural one is the Lagrangian distribution coming from the Lagrangian bundle structure of $(X,\omega)\to B$ and call it with $\Delta$. In the case of having an SU$(3)$-structure, as it will be for us, we can use the metric $g$ of the structure to define the $g$-orhtogonal complement $\Delta^\perp$ of $\Delta$.  
%This allow us to decompose the space of differential forms:

%$$\mathcal{A}^\bullet(X)=\bigoplus_{p+q}\mathcal{A}^{p,q}_\Delta(X)$$
%
%where $\mathcal{A}^{p,q}_\Delta$ ranges over $p$ $\Delta$-directions and over $q$ $\Delta^\perp$-directions. 

%With this in mind, denote with $\mathcal{A}^{p,q}_{B,\Delta}(X)\subset\mathcal{A}^{p,q}_{\Delta}(X)$ the space of \textbf{semi-flat $(p,q)$-forms}. Since in our construction the choice of the Lagrangian distribution will be indeed induced by the fibration itself we will omit the subscript for the distribution and we will write just  $\mathcal{A}^{p,q}_{B}(X)$. Clearly the p-directions in $\Delta$ correspond to the  $d\theta_1,d\theta_2,d\theta_3$ while the q-directions, relative to $\Delta^\perp$, to $dr_1,dr_2,dr_3$.

Analogously we will denote by $\mathcal{A}^{p,q}_B(\check X)$ the semi-flat $(p,q)$-forms on the SYZ-dual $\check X$ which are locally written as: 

$$\check\phi=\sum_{I,J}a_{IJ}(r) dz_I\w d\bar z_J$$

In the 3-dimensional case of we can refine the previous definition of SU$(3)$-structures:

\begin{definition}An $SU(3)$-manifold $(X,\omega,\Omega)$ is said to be semi-flat  supersymmetric of type IIA if $d\omega=0$ and $d\, \rm{Re}\,\Omega = 0$ and both $\omega$ and $\Omega$ are in $\mathcal{A}_B^\bullet(X,\mathbb{C})$ 
\end{definition}
\begin{definition} An $SU(3)$-manifold $(X,\omega,\Omega)$ is said to be semi-flat  supersymmetric of type IIB if $d(\omega^2) = 0$ and $d\Omega = 0$ and both $\omega$ and $\Omega$ are in $\mathcal{A}_B^\bullet(X,\mathbb{C})$.
\end{definition}

%On the complex side, the $(p,q)$-decomposition on forms is taken with respect to the complex polarization induced by the aforementioned complex structure $\check\Omega$. 
%However, on $\check X$, we can take another polarization induced by the dual action-angle coordinates $\{r_i,\check\theta_i\}$. 

We further define the {\em polarization switch operator $\mathcal{P}$} on $\mathcal{A}_B^{\bullet,\bullet}(\check X)$
%as the operator which acts as a switch on the basic wedges 
in the following way:

%$$ dz_I\longleftrightarrow d\check\theta_I    \ \ \ , \ \ \ d\bar z_J\longleftrightarrow dr_J$$
%
%%This extra structure in the complex side allow us to define a new operator , namely the \textbf{polarization switch operator $\mathcal{P}$} on $\mathcal{A}_B^{\bullet,\bullet}(\check M)$, as the operator which acts as a switch on the basic wedges in the following way:
%
%$$ dz_I\longleftrightarrow d\check\theta_I    \ \ \ , \ \ \ d\bar z_J\longleftrightarrow dr_J$$
%Therefore 

$$\check\phi=\sum_{I,J}a_{IJ}(r) dz_I\w d\bar z_J \, \mapsto \, \mathcal{P}\cdot\check\phi=\sum_{I,J}a_{IJ}(r) d\check\theta_I\w dr_J\,.$$

\subsection{The Fourier-Mukai transform}
Let $(X,\omega)$ and $(\cc X,\cc\Omega)$ be a semi-flat mirror pair on a $n$-dimensional base $B$. Consider their fiber product over $B$:

\begin{center}
	\begin{tikzcd}[column sep=scriptsize]
		& X\times_B\check X \arrow[dl,"p"'] \arrow[dr,"\check p"] \\
		X \arrow[dr,"\pi"'] 
		& & \check X \arrow[dl,"\check\pi"] \\
		& B 
	\end{tikzcd}
\end{center}

On the Poincar\'e line bundle over $X\times_B\check X$ there is a universal connection which locally is written as $d+i\big(\check\theta_kd\theta_k+\theta_kd\check\theta_k\big)$. Its curvature form is

\begin{equation}
	F=2i\sum_i^3 d\check\theta_i\wedge d\theta_i \,.
\end{equation}

Let $\phi\in\mathcal{A}^\bullet_B(X)$ and $\check\phi\in\mathcal{A}_B^\bullet(\check X)$. Their {\em Fourier-Mukai transforms} are defined as

\begin{equation}
	\begin{split}
		&\text{FT}\cdot \check\phi:=p_\ast\Big((\check p^\ast(\mathcal{P}\cdot\check\phi))\wedge\exp{\frac{F}{2i}}\Big)\\
		&\text{FT}\cdot \phi:=\mathcal{P}^{-1}\cdot\Big(\check p_\ast\big((p^\ast\phi)\wedge\exp{\frac{-F}{2i}}\big)\Big)\,,\\
	\end{split}
\end{equation}

where the pushforward maps $p_\ast,\check p_\ast$ are just the integration along the fibers.

%{\color{red} We recall now the main results of the paper by Lau, Tseng, and Yau exploiting the Fourier-Mukai transform:}
The main results by Lau, Tseng, and Yau exploiting the Fourier-Mukai transform involve a refined version of the symplectic cohomology developed by Tseng and Yau in \cite{TY1,TY2}.

Let $\Delta \subset TM$ be the Lagrangian distribution coming from the Lagrangian bundle structure of $(X,\omega)\to B$. In the presence of a metric we have also its orthogonal $\Delta^\perp$. This choice allow us to decompose the space of differential forms:
%Since we want to work with $(p,q)$-forms also on the symplectic side, it is necessary to choose an integrable distribution to split the (co)tangent bundle. The most natural one is the Lagrangian distribution coming from the Lagrangian bundle structure of $(X,\omega)\to B$ and call it with $\Delta$. In the case of having an SU$(3)$-structure, as it will be for us, we can use the metric $g$ of the structure to define the $g$-orhtogonal complement $\Delta^\perp$ of $\Delta$.  
%This allow us to decompose the space of differential forms:

$$\mathcal{A}^\bullet(X)=\bigoplus_{p+q}\mathcal{A}^{p,q}_B(X)$$

where $\mathcal{A}^{p,q}_B$ ranges over $p$ $\Delta$-directions and over $q$ $\Delta^\perp$-directions. 

The $\Delta$-refined Tseng-Yau cohomology of $(X,\omega)$ is
$$ H^{p,q}_{B, TY}(X) := \frac{{\rm Ker} (d + d^\Lambda) \cap {\mathcal A}^{p,q}_B(X)}{{\rm Im}(d d^\Lambda) \cap {\mathcal A}_B^{p,q}(X)} $$

where $d^\Lambda = d\Lambda -\Lambda d$ and $\Lambda$ is the adjoint of the Lefschetz operator $L=\omega \wedge \cdot\,.$ 
%With this in mind, denote with $\mathcal{A}^{p,q}_{B,\Delta}(X)\subset\mathcal{A}^{p,q}_{\Delta}(X)$ the space of \textbf{semi-flat $(p,q)$-forms}. Since in our construction the choice of the Lagrangian distribution will be indeed induced by the fibration itself we will omit the subscript for the distribution and we will write just  $\mathcal{A}^{p,q}_{B}(X)$. Clearly the p-directions in $\Delta$ correspond to the  $d\theta_1,d\theta_2,d\theta_3$ while the q-directions, relative to $\Delta^\perp$, to $dr_1,dr_2,dr_3$.

We also recall that the (semi-flat) Bott-Chern cohomology of $(\cc X,\cc \Omega)$ is 
$$ H^{p,q}_{B, BC}(\cc X) := \frac{{\rm Ker}\, d \, \cap {\mathcal A}^{p,q}_B(\cc X)}{{\rm Im}(\partial \bar\partial) \cap {\mathcal A}_B^{p,q}(\cc X)} \,.$$

\begin{theorem}[Theorem 4.5 and Theorem 6.7 \cite{LTY}]
		Fourier-Mukai transform is an isomorphism of double complexes
	
	$$\Big(\mathcal{A}^{\bullet}_B(X,\mathbb{C}),\frac{(-1)^ni}{2}d,\frac{(-1)^ni}{2}d^\Lambda\Big)\simeq\Big(\mathcal{A}^{\bullet}_B(\check X,\mathbb{C}),\bar\partial,\partial\Big)$$
	and at level of cohomologies gives 
	\begin{equation}\label{mirrornumbers}
		H^{n-p,q}_{B,TY}(X,\mathbb{C})\simeq H^{p,q}_{B,BC}(\check X)\,.
	\end{equation}
\end{theorem} 

%\begin{theorem}[Theorem 5.1 \cite{LTY}]
%	Let $(X,\omega)$ and $(\check X,\check \Omega)$ a semi-flat mirror pair. Let $\check\omega$ be a real $(1,1)$-form in $\mathcal{A}^{1,1}_B(\check X)$ and set $\Omega=FT(e^{2\check\omega})$. Then
%	
%	\begin{enumerate}
%		\item The triple $(\check X,\check\omega,\check\Omega)$ forms a SU$(n)$-structure if and only if $(X,\omega,\Omega)$ forms a SU$(n)$-structure. Moreover the conformal factors are related by the relation $F\check F=2^{2n}$;
%		\item $(X,\omega,\Omega)$ is supersymmetric of type IIA if and only if $(\check X,\check\omega,\check\Omega)$ is supersymmetric of type IIB;
%		\item Under Fourier-Mukai transform the fluxes $\rho_A$ and $\rho_B$ correspond to each other up to a constant multiple.
%	\end{enumerate}
%	\end{theorem}
\begin{theorem}[Theorem 5.1 \cite{LTY}]
	\label{LTYth2}
	Let $(X,\omega)$ and $(\check X,\check \Omega)$ be a $3$-dimensional semi-flat mirror pair. Let $\check\omega$ be a real $(1,1)$-form in $\mathcal{A}^{1,1}_B(\check X)$ and set $\Omega=FT(e^{2\check\omega})$. Then
	
	\begin{enumerate}
		\item The triple $(\check X,\check\omega,\check\Omega)$ defines a SU$(3)$-structure if and only if $(X,\omega,\Omega)$ defines a SU$(3)$-structure. 
		%Moreover the conformal factors are related by the relation $F\check F=2^{2n}$;
		\item $(X,\omega,\Omega)$ is supersymmetric of type IIA if and only if $(\check X,\check\omega,\check\Omega)$ is supersymmetric of type IIB.
		%\item Under Fourier-Mukai transform the fluxes $\rho_A$ and $\rho_B$ correspond to each other up to a constant multiple.
	\end{enumerate}
\end{theorem}

Note that we stated the last Theorem in the $3$-dimensional case since in this paper we are not dealing with general type IIA and type IIB $\rm SU(n)$-structures.

\section{Known symplectic half-flat manifolds}

Few non-K\"ahler symplectic half-flat manifolds are known.
Main examples are nilmanifolds and solvmanifolds. 

The first nilpotent example was found in \cite{TesiB}, the case of nilpotent Lie algebras is considered in \cite{ContiTomassini}, while the case of solvable non-nilpotent Lie algebras is treated in \cite{FMOU}. Important contributions with explicit examples are \cite{deBaTomma1}, \cite{deBaTomma2} and \cite{TommaLuigi}.
%For each Lie algebra in the list we {\em list} the known symplectic half flat structures. 
As far as we know there is still no {\em classification} up to isomorphism of symplectic half flat structures on any solvable Lie algebra.

The following is the complete list of non-abelian unimodular solvable Lie algebras admitting invariant symplectic half flat structures.  
We also specify if the algebra is completely solvable. 

Following the usual convention we present a Lie algebra choosing a left-invariant coframe $e^1,\dots,e^6$ and listing their differential. As usual $e^{ij}$ stands for $e^{i} \wedge e^j$.

\bigskip

%\scalemath{0.7}{
%\begin{table}[H]
%	\label{algebreIIA}
\begin{center}
	 \captionof{table}{Lie algebras with left-invariant type IIA structure \label{algebreIIA}}
\centering
\renewcommand{\arraystretch}{1.5}	
\scalebox{0.7}{
	\begin{tabular}{|c|l|l|l|r|}
		\hline
		 & Lie algebra & $\omega$ & $\Omega$ & c.s. \\
		\hline
		1 & $(0,0,0,0,e^{12},e^{13})$ & $e^{14}+e^{26}+e^{35}$ & $\big(e^1+ie^4\big)\w \big(e^2+ie^6\big)\w \big(e^3+ie^5\big)$ & yes\\
		\hline
		2 & $(0,0,0,e^{12},e^{13},e^{23})$ & $e^{61}+\lambda e^{52}+(1-\lambda)e^{34}$ & $\big(e^6+ie^1\big)\w \big(e^5+i\lambda e^2\big)\w \big(e^3+i(1-\lambda)e^4\big)$& yes \\
		\hline
		3 & $(0,-e^{13},-e^{12},0,-e^{46},-e^{45})$ & $e^{14}+e^{23}+e^{56}$ & $(1+i)\big(e^1+ie^4\big)\w \big(e^2+ie^3\big)\w \big(e^5+ie^6\big)$ & yes \\
		\hline
		4 & $(e^{15},-e^{25},-e^{35},e^{45},0,0)$ & $e^{31}+e^{24}+e^{56}$ & $\big(e^3+ie^1\big)\w \big(e^2+ie^4\big)\w \big(e^5+ie^6\big)$ & yes \\
		\hline
		5 & $(\alpha e^{15}+e^{25},- e^{15}+\alpha e^{25},-\alpha e^{35}+e^{45},-e^{35}-\alpha e^{45},0,0)$ &$e^{13}+e^{24}+e^{56}$ & $\big(e^3+ie^1\big)\w \big(e^2+ie^4\big)\w \big(e^5+ie^6\big)$ & no \\
		\hline
		6 & $(e^{23},- e^{36},e^{26},e^{26}-e^{56},e^{36}+e^{46},0)$ & $-2e^{16}+e^{34}-e^{25}$ & $\big(-2e^1+ie^6\big)\w \big(e^3+ie^4\big)\w \big(e^5+ie^2\big)$ & no \\
		\hline
		7 & $(e^{16}+e^{35},- e^{26}+ e^{45},e^{36},-e^{46},0,0)$ &$e^{14}+e^{23}+e^{56}$ & $\big(e^1+ie^4\big)\w \big(e^2+ie^3\big)\w \big(e^5+ie^6\big)$ & yes \\
		\hline
		8 & $(-e^{16}+e^{25},-e^{15}- e^{26},e^{36}-e^{45},e^{35}+ e^{46},0,0)$ &$e^{14}+e^{23}+e^{65}$ & $\big(e^1+ie^4\big)\w \big(e^2+ie^3\big)\w \big(e^6+ie^5\big)$ & no \\
		\hline
	\end{tabular}   
}
%	\caption{Lie algebras with left-invariant type IIA structure.}
	\renewcommand{\arraystretch}{1}
%\end{table}
\end{center}
%}

%\begin{table}[H]
%\centering
%\label{typeIIAalgebras}
%\begin{tabular}{|l|l|l|r|}
%\hline
%Lie algebra & $\omega$ & $\Omega$ & c.s. \\
%\hline
%$(0,0,0,0,e^{12},e^{13})$ & ... & ... & yes\\
%\hline
%$(0,0,0,e^{12},e^{13},e^{23})$ & ... & ... & yes \\
%\hline
%$(0,-e^{13},-e^{12},0,-e^{46},-e^{45})$ & ... & ... & yes \\
%\hline
%$(e^{15},-e^{25},-e^{35},e^{45},0,0)$ & ... & ... & yes \\
%\hline
%$(\alpha e^{15}+e^{25},- e^{15}+\alpha e^{25},-\alpha e^{35}+e^{45},-e^{35}-\alpha e^{45},0,0)$ & ... & ... & no \\
%\hline
%$(e^{23},- e^{36},e^{26},e^{26}-e^{56},e^{36}+e^{46},0)$ & ... & ... & no \\
%\hline
%$(e^{16}+e^{35},- e^{26}+ e^{45},e^{36},-e^{46},0,0)$ & ... & ... & yes \\
%\hline
%$(-e^{16}+e^{25},-e^{15}- e^{26},e^{36}-e^{45},e^{35}+ e^{46},0,0)$ & ... & ... & no \\
%\hline
%\end{tabular}   
%\caption{Lie algebras with left-invariant type IIA structure.}
%\end{table}

%First example is the nilmanifold arising from the Lie algebra $\mathbf{(0,0,0,0,12,13)}$.
%The relevant nilpotent Lie group is given by matrices of the form
%$$
%\begin{bmatrix}
%1 & x_1 & x_2 & x_3 & 0 & 0 \\
%0 & 1 & x_4 & x_5 & 0 & 0 \\
%0 & 0 & 1 & 0 & 0 & 0 \\
%0 & 0 & 0 & 1 & 0 & 0 \\
%0 & 0 & 0 & 0 & 1 & x_6 \\
%0 & 0 & 0 & 0 & 0 & 1 \\
%\end{bmatrix}
%$$ 
%The second example is the nilmanifold arising from the Lie algebra $\mathbf{(0,0,0,12,13,23)}$.
%The relevant nilpotent Lie group $G$ is given by matrices of the form
%$$
%\begin{bmatrix}
%1 & x_1 & x_5 & 0 & 0 & x_4 \\
%0 & 1 & x_3 & 0 & 0 & x_2 \\
%0 & 0 & 1 & 0 & 0 & 0 \\
%0 & x_2 & x_6 & 1 & x_2 & 0 \\
%0 & 0 & 0 & 0 & 1 & -x_2 \\
%0 & 0 & 0 & 0 & 0 & 1 \\
%\end{bmatrix}
%\,.
%$$
More compact examples are provided by twistor spaces of compact 4-dimensional self-dual Einstein manifolds of negative scalar curvature, see
\cite{Xuthesis}.
For interesting non-compact examples see \cite{PodeRaf1, PodeRaf2}.

\section{Solvmanifolds, semi-direct products and equivariant dual torus bundles}
\label{construction}
Let $G$ be a simply connected $n$-dimensional Lie group endowed with a complete  left-invariant affine structure with developing map $D: G \to \R^n$. Without loss of generality we may assume $D(1_G)=0$. 
%Completeness of the affine structure together with simply connectedness of $G$ implies that $D$ is a diffeomorphism.\\

Note that such a group is necessarily {\em solvable}. Indeed left invariant complete affine structures on Lie groups correspond to simply transitive subgroups of  ${\rm Aff}(\R^n)$ (see \cite{FG}) and these are solvable (see \cite{Aus}).

%Let $G$ be $\R^n$ endowed with a group structure making it into a solvable Lie group {\color{red}and} such that the left-multiplications are 
%affine transformations. 
%This is essentially a simply connected $n$-dimensional solvable Lie group with a {\em left-invariant} affine structure.\\

Let us call $\alpha$ the faithful affine representation $\alpha: G \to {\rm Aff}(\R^n)$ given by $\alpha(g)=D \circ L_g \circ D^{-1}$ where $L_g: G \to G$ is the left-multiplication by $g$. Let $\rho: G \to {\rm GL}(n,\R)$ be its linear part. Of course $\rho$ need not be faithful.

%Assume now that 
Choosing a lattice (i.e.\ a co-compact discrete subgroup) $\Gamma \leqslant G$ whose left multiplications, read through $D$, are {\em integral} affine gives a well defined integral affine structure $\mathcal A$ to the set of right cosets $B = G/\Gamma$. Of course, for such a $\Gamma$ to exist, the group $G$ needs to be unimodular.
The holonomy representation of this structure is the restriction of $\alpha$ to $\Gamma$.

According to the construction defined in section \ref{affine} we thus have a well defined lattice $\Lambda^* \subset T^*B$, a Lagrangian torus fibration $X=T^*B / \Lambda^* \to B$ and its dual torus fibration $\cc X=TB / \Lambda \to B$.
We will work for simplicity on the latter.
%\subsection{The $TB$ side}
%\subsection{Equivariant dual torus bundles}
Let $\pi: G \to G/\Gamma$ be the canonical projection.
Let us identify $B \times \R^n$ with $TB$ via the map $$(\Gamma h,v) \mapsto d(\pi \circ L_h \circ D^{-1})_{0}\,v\,.$$ Using this identification we define an action of $G \ltimes_{\rho} \R^n$ on $TB$ by
$$(g,y)(\Gamma h,v)=(\Gamma hg^{-1},y+\rho(g)v).$$
%{\color{red}
%Verifica
%$$(g',y')(g,y)(\Gamma h,v)=(g',y')(\Gamma hg^{-1},y+\rho(g)v)=(\Gamma hg^{-1}g'^{-1},y'+\rho(g')y+\rho(g'g)v)$$
%$$(g'g,y'+\rho(g')y)(\Gamma h,v)=(g',y')(\Gamma hg^{-1}g'^{-1},y'+\rho(g')y+\rho(g'g)v)$$
%}
The lattice $\Lambda \subset TB$ is defined as follows:
$$\Lambda_{\Gamma h}=\Z\{d(\pi \circ D^{-1})_{D(h)} e_i: i=1,\ldots,n\}$$
Of course here $e_i$ is thought of as an element of $T_hG=T_h \R^n=\R^n$.
Note that the lattice is well defined exactly because $\rho(\gamma )\in {\rm GL}(n,\Z)$ for every $\gamma \in \Gamma$.

Now we claim that the previous action descends to the quotient $TB/\Lambda$.

In order to prove it we must show that for every $g \in G$, $y \in \R^n$, $h\in G$, $w \in T_{\Gamma h}B$ and $\lambda \in \Lambda_{\Gamma h}$ we have
$$(g,y)(w+\lambda)-(g,y)w \in \Lambda_{\Gamma hg^{-1}}.$$
Now let $m=(m_1,\ldots,m_n) \in \Z^n$:
\begin{eqnarray*} 
(g,y)(w+\lambda) & = & (g,y)(d(\pi \circ L_h \circ D^{-1})_{0}\, v + \sum_i m_i d(\pi\circ D^{-1})_{D(h)} e_i) \\
& = &  (g,y)(d(\pi \circ L_h\circ D^{-1})_0\, (v + \sum_i m_i (d(D \circ L_{h^{-1}}\circ D^{-1})_{D(h)} e_i)) \\
& = &  d(\pi \circ L_{hg^{-1}}\circ D^{-1})_{0} (y + \rho(g) v + \sum_i m_i d(\alpha(gh^{-1}))_{D(hg^{-1})} e_i) \\
& = &  d(\pi \circ L_{hg^{-1}}\circ D^{-1})_{0} (y + \rho(g) v + \sum_i m_i (d(D \circ L_{gh^{-1}}\circ D^{-1}))_{D(hg^{-1})} e_i) \\
%& = &  d(\pi \circ L_{hg^{-1}})_{1_G} (y + \rho(g) v) + \sum_i m_i d\pi_{hg^{-1}} e_i \\
& = &  (g,y)w + \sum_i m_i d(\pi \circ D^{-1})_{D(hg^{-1})} e_i \,.
%& (g,y)(d\pi_h (dL_h)_{1_G} v + d\pi_h e_i) = \\
%& = &  d\pi_{hg^{-1}} (dL_{hg^{-1}})_{1_G} (y + (dL_g) v) + d\pi_h e_i
\end{eqnarray*}
Finally we note that the action $G \ltimes_{\rho} \R^n \curvearrowright TB/\Lambda$ is clearly transitive, and the stabilizer at $(\Gamma,0)$ is exactly $\Gamma \ltimes_{\rho} \Z^n$.\\
Dualizing everything we obtain the following
\begin{theorem}
\label{main}
Let $X\to G/\Gamma$ be the torus bundle induced by the integral affine structure defined by the triple $(G, \Gamma, D)$.
Let $\check X\to G/\Gamma$ be its dual.  Then
\begin{enumerate}
\item The total space $X$ is acted on transitively by the semidirect product $G \ltimes_{\rho^*} (\R^n)^*$ with stabilizer $\Gamma \ltimes_{\rho^*} (\Z^n)^*$, where $\rho^*:G \to {\rm Aff}((\R^n)^*)$ is the dual representation induced by $\rho$. 
\item The total space $\check X$ is acted on transitively by the semidirect product  $G \ltimes_{\rho} \R^n$with stabilizer $\Gamma \ltimes_{\rho} \Z^n$.
\end{enumerate}
\end{theorem}
\begin{rem}
	\label{remGamma}
	{\rm
	From the argument preceding Theorem \ref{main} it is apparent that the same construction can be applied in the slightly more general case in which $\alpha_{|\Gamma}$ lies in the automorphism group of a lattice $\Xi$ in $\R^n$ conjugated to $\Z^n$. In this case the stabilizers mentioned in Theorem \ref{main} will be  $\Gamma \ltimes_{\rho^*} \Xi^*$ and  $\Gamma \ltimes_{\rho} \Xi$.
}
\end{rem}
\subsection{${\bf SU}(n)$-manifolds from affine structures}
\label{distinguished}
As above we assume that $G$ is a simply connected $n$-dimensional unimodular Lie group endowed with a complete  left-invariant affine structure with developing map $D: G \to \R^n$. 

Let us endow $T^*G$ with the Lie group structure induced by the identification $T^*G = G \ltimes_{\rho^*} (\R^n)^*$.
% as in the section above.
%We want to show that when  $T^*G$ endowed with the Lie group structure induced by the identification $T^*G = G \ltimes_\rho \R^n$ carries a natural left-invariant SU(n)-structure.\\
Let $\Gamma \leqslant G$ be a lattice whose left multiplications, read through $D$, are {\em integral} affine. This gives the coset space $B=G/\Gamma$ an integral affine structure.
Thus we have a well defined {\em fiberwise lattice} $\Lambda^* \subset T^*B$.

%Notice that this just means that for every $b\in B$ the set $\Lambda^*_b = \Lambda^* \cap \pi^{-1}(b)$ is a lattice of $T^*_bB$ and a priori has nothing to do with any global group structure on $T^*G$ (Here $\pi: T^*B \to B$ is the quotient).\\
The canonical symplectic structure $\omega$ on $T^*G$ passes to the quotient $X=T^*B/\Lambda^*$ and the induced  projection $X \to B$ (that we are going to call again $\pi$) becomes a Lagrangian torus bundle.

%Let $\omega$ be the canonical symplectic structure on $T^*G$.
%This is easily verified to be left-invariant.
%Indeed...\\
%Now let $q_1, q_2, q_3$ be local affine coordinates on $G$ and $q_1, q_2, q_3, p_1,p_2,p_3$ be Darboux coordinates on $T^*G = G \ltimes_\rho \R^3$ so that
%$$\omega = \sum_i dp_i \wedge dq_i\,.$$

In the previous section we proved that $X$ itself is a solvmanifold under the action of $T^*G = G \ltimes_{\rho^*} (\R^n)^*$. 
%\begin{theorem}
%The canonical symplectic form $\omega$ on $X=(G \ltimes_{\rho} \R^n)/(\Gamma \ltimes_{\rho} \Z^n)$ is left-invariant.
%\end{theorem}
\begin{prop}
The canonical symplectic form $\omega$ on $T^*G=G \ltimes_{\rho^*} (\R^n)^*$ is left-invariant.
\end{prop}
\begin{proof}
%Let us call $\Xi=\Gamma \ltimes_{\rho} \Z^n$. Let $1=\Xi (e,0)$.
%The proof has to be intrinsic.
%Let $h=(g,v) \in G \ltimes_{\rho} \R^n$. 
%We want to prove that $\omega _{h}((dL_h)_1 u_1,(dL_h)_1 u_2) = \omega _{1}(u_1,u_2)$ for every $u_1,u_2 \in T_1X$.\\
The important fact is that the local action-angle coordinates $r_1, \dots, r_n, \theta_1,\dots,\theta_n$ coming from Arnol'\!d-Liouville theorem become globally defined functions once lifted to the universal cover $\tilde X$ which coincide with $G \ltimes_{\rho} \R^n$.\\
On $\tilde X$ globally $\omega = \sum_i d\theta_i \wedge dr_i$. \\
Let $h=(g,v) \in G \ltimes_{\rho^*} (\R^n)^*$. 
%Let $\mathcal L_h$ be the matrix representing $dL_h^*$ with respect to the basis $(dr_1)_h, \dots, (dr_n)_h, (d\theta_1)_h,\dots,(d\theta_n)_h$ and  $(dr_1)_{(e,0)}, \dots, (dr_n)_{(e,0)}, (d\theta_1)_{(e,0)},\dots,(d\theta_n)_{(e,0)}$. \\
%We need to prove that $\mathcal L_h$ is a symplectic matrix.
Let $\mathcal L_g$ be the $n \times n$ matrix representing $(L_g)_*$ in the global frame $\frac{\partial}{\partial r_1},\dots,\frac{\partial}{\partial r_n}$ of $G=\tilde B$.
Now from the definition of the group law on $G \ltimes_{\rho^*} (\R^n)^*$ one gets $L_h^*(dr_i) = \sum_j (\mathcal L_g)^{-1}_{ji} dr_j$ and $L_h^*(d\theta_i) = \sum_j (\mathcal L_g)_{ij} d\theta_j$.
From this the result immediately follows since the matrix representing $L_h^*$ w.r.t. $dr_1,\dots,dr_n,d\theta_1,\dots,d\theta_n$ is of the kind 
$
\begin{bmatrix} 
 (A^{T})^{-1}&  0\\  
0 & A
\end{bmatrix} 
$
hence symplectic.
%Let $\mathcal L_h$ be the matrix representing $dL_h^*$ with respect to the basis $(dr_1)_h, \dots, (dr_n)_h, (d\theta_1)_h,\dots,(d\theta_n)_h$ and \\ $(dr_1)_{(e,0)}, \dots, (dr_n)_{(e,0)}, (d\theta_1)_{(e,0)},\dots,(d\theta_n)_{(e,0)}$. \\
%We need to prove that $\mathcal L_h$ is a symplectic matrix.
%By Arnol'd-Liouville theorem there exists $U$ neighborhood of $e$ in $B$ and $r_1, \dots, r_n \in C^\infty(U)$ and $\theta_1,\dots,\theta_n \in C^$
\end{proof}
%Let $\Omega$ be the complex left-invariant $n$-form on $G \ltimes_{\rho^*} (\R^n)^*$ which coincides with $\bigwedge_{k=1}^n (d\theta_k + idr_k)$ at the identity.
%This gives $G \ltimes_{\rho} \R^n$ a left-invariant symplectic SU$(n)$-structure which of course passes to the quotient $X=(G \ltimes_{\rho^*} (\R^n)^*)/(\Gamma \ltimes_{\rho^*} (\Z^n)^*)$.\\

%From now on we make the further assumption that  the affine structure on $B$ is {\em special}, i.e. the linear part of the transition functions lies in ${\rm SL}(n,\Z)$. 

Dualizing each ingredient 
%according to the construction of section \ref{construction} and taking into account that the 
%affine structure is {\em special} 
we get a left-invariant holomorphic structure on $TG = G \ltimes_{\rho} \R^n$ hence on 
$\check X=(G \ltimes_{\rho} \R^n)/(\Gamma \ltimes_{\rho} \Z^n)$.
\begin{prop}
The tangent bundle $TG=G \ltimes_{\rho} \R^n$ has a canonical left-invariant integrable complex structure.
\end{prop}
\begin{proof}
On $TG$ we have the global coordinates $r_1,\dots,r_n,\check \theta_1,\dots,\check \theta_n$ where the $\check\theta_k$'s are the dual coordinates corresponding to the $\theta_k$'s. The coordinates $z_k=\check \theta_k+i r_k$, $k=1,\dots,n$ give $TG$ an integrable complex structure. Now we prove that the complex volume $\check \Omega = \bigwedge_{k=1}^n dz_k$ is left-invariant on $G \ltimes_{\rho} \R^n$.\\
Let $\check h=(g,\check v) \in G \ltimes_{\rho} \R^n$. 
From the definition of the group law on $G \ltimes_{\rho^*} \R^n$ one gets  $L_h^*(d\check\theta_i) = \sum_j (\mathcal L_g)^{-1}_{ji} d\check \theta_j$.
Thus the matrix representing $L_h^*$ w.r.t. $dr_1,\dots,dr_n,d\check\theta_1,\dots,d\check\theta_n$ is of the kind 
$
\begin{bmatrix} 
 (A^{T})^{-1}&  0\\  
0 &  (A^{T})^{-1}
\end{bmatrix} 
$
hence complex. Moreover this lies in SL$(n,\C)$ since every complete affine structure on a compact solvmanifold is in fact {\em special}, that is the linear part of the transition functions lies indeed in ${\rm SL}(n,\Z)$ (see \cite[Theorem 2]{FGH}). 
\end{proof}

The left-invariant symplectic structures on $X$ and the left-invariant complex structure on $\cc X$ defined above only depend on the affine structure and not on the choice of $D$.
Though the choice of the developing map $D$ allows us also to define distinguished left-invariant ${\rm SU}(n)$-structures on $X$ and $\cc X$. 

Let us start from the symplectic side $X$. 
For $i=1\dots,n$ define $e_i$ to be the global left-invariant 1-form on $T^*G = G \ltimes_{\rho^*} (\R^n)^*$ such that ${e_{k}}_{|e}={d\theta_k}_{|e}$ and ${e_{k+n}}_{|e}={dr_k}_{|e}$.
Note that in these coordinates the canonical symplectic form is $\omega = \sum_k e_{k} \wedge e_{n+k}$. Now set $\Omega:= \bigwedge_{k=1}^n (e_{k} + ie_{k+n})$. Clearly $(\omega,\Omega)$ defines a symplectic $\rm{SU}(n)$-structure on $X$.

The construction is analogous on the complex side.
For $i=1\dots,n$ define $\cc e_i$ to be the global left-invariant 1-form on $TG = G \ltimes_{\rho} \R^n$ such that ${\cc {e_k}}_{|e}={d\cc\theta_k}_{|e}$ and ${\cc e}_{k+n}{}_{|e}={dr_k}_{|e}$.
Note that in these coordinates the canonical complex $n$-form takes the expression  $\cc\Omega:= \bigwedge_{k=1}^n (\cc e_{k} + i\cc e_{k+n})$.\\
Now set $\cc \omega:= \sum_{k=1}^n (\cc e_{k} \wedge \cc e_{k+n})$. Clearly $(\cc \omega,\cc \Omega)$ defines a complex  $\rm{SU}(n)$-structure on $\cc X$.\\

In dimension 3 we have the following lemma that gives the link with Theorem \ref{LTYth2}.
\begin{lemma}
Let $(X,\omega)$ and $(\cc X, \cc\Omega)$ be the 3-dimensional semi-flat mirror pair induced by $(G,\Gamma,D)$. Let $\cc \omega$ be the $(1,1)$ form on $\cc X$ as above. Then the form $\Omega$ defined above is the Fourier-Mukai transform of $e^{2\cc \omega}$.
\end{lemma}
\begin{proof} It is enough to express the two relevant  forms in action-angle coordinates. First set $S=\mathcal{L}^{-T}\mathcal{L}^{-1}$ and $\eta_j= \sum_k S_{jk} dr_k$. 
	
Let us also introduce the following basis of $(1,0)$-forms $\psi^k=\cc e^{k} +i\cc e^{k+3}$. They are related to the differential of complex coordinates via $\psi^k=\sum_{k=1}^3 \mathcal{L}^{-1}_{kj}dz_j=\sum_{k=1}^3 \mathcal{L}^{-1}_{kj}(d\cc\theta_j+idr_j)$
	Then we have
$$
\begin{aligned}
\cc\omega & =\frac{i}{2}\sum_{i=1}^3 \psi^{k\bar k}=\frac{i}{2}\sum_{k=1}^3\big(dz_k \w S_{kj}d\bar z_j\big)=\sum_{k=1}^3d\cc\theta_k\w\eta_k\\
%\cc \omega=\sum_k d\cc \theta_k \wedge \eta_k\\
\Omega & = \bigwedge_k (d\theta_k + i\eta_k)
\end{aligned}
$$

Now we can do the following straightforward computation
%Let us also introduce the following basis of $(1,0)$-forms $\psi^k=\cc e^k+i\cc e^{k+3}$. They are related to the differential of complex coordinates via $\psi^k=\sum_{k=1}^3 \mathcal{L}^{-1}_{kj}dz_j=\sum_{k=1}^3 \mathcal{L}^{-1}_{kj}(d\cc\theta_j+idr_j)$.  Then 
%we have also
%%the $(1,1)$-form $\cc\omega$ can be taken as 
%$$\cc\omega=\frac{i}{2}\sum_{i=1}^3 \psi^{k\bar k}=\frac{i}{2}\sum_{k=1}^3\big(dz_k \w S_{kj}d\bar z_j\big)=\sum_{k=1}^3d\cc\theta_k\w\eta_k$$.
%\begin{equation}
%	\begin{split}
%		\cc\omega &= \frac{i}{2}\sum_{k=1}^{3}\psi^{k\bar k}\\
%		&=\frac{i}{2}\sum_{k=1}^{3}\Big(\mathcal{L}^{-1}_{kj}dz_j\w \mathcal{L}^{-1}_{kj}d\bar z_j\Big)\\
%		&=\frac{i}{2}\sum_{k=1}^{3}\Big(\mathcal{L}^{-1}_{kj}(d\cc\theta_j+idr_j)\w \mathcal{L}^{-1}_{kj}(d\cc\theta_j-dr_j)\Big)\\
%		&=\frac{i}{2}\sum_{k=1}^{3}\Big((\mathcal{L}^{-1}_{kj}d\cc\theta_j+i\mathcal{L}^{-1}_{kj}dr_j)\w (\mathcal{L}^{-1}_{kj}d\cc\theta_j-\mathcal{L}^{-1}_{kj}dr_j)\Big)\\
%		&=\sum_{k=1}^3 \Big(\mathcal{L}^{-1}_{kj}d\cc\theta_j \w  \mathcal{L}^{-1}_{kj}dr_j\Big)\\
%		&=\sum_{k=1}^3 \Big(d\cc\theta_k\w S_{jk}dr_j\Big)\\
%		&=\sum_{k=1}^3 \Big(d\cc\theta_k\w\eta_k\Big)\\
%		\\	\end{split}
%\end{equation}
\begin{equation*}
	\begin{split}
		FT(e^{2\cc\omega}) & =p_\ast\Big(p^\ast\big(\mathcal{P}\cdot e^{2\cc\omega}\big)\w e^{\frac{F}{2i}}\Big)\\
		& =p_\ast\Big(e^{i\sum_{k=1}^3d\cc\theta_k\w\eta_k}\w e^{\sum_{k=1}^3d\cc\theta_k\w d\theta_k}\Big)\\
		&=p_\ast\Big(e^{\sum_{k=1}^3d\cc\theta_k\w(d\theta_k+i\eta_k)}\Big)\\
		&=p_\ast\Big(\bigwedge_{k=1}^3 d\cc\theta_k \w (d\theta_k+i \eta_k)\Big)\\
		&=\bigwedge_{k=1}^3\big(d\theta_k+i\eta_k\big)=\Omega\,.
	\end{split}
\end{equation*}
\end{proof}

%Now we can do the following straightforward computation
%\begin{equation*}
%\begin{split}
%FT(\Omega) & =\mathcal{P}^{-1}\Bigg(\check\pi_\ast\bigg(\Big(\bigwedge_{i=1}^3 (d\theta_i+i\eta_i)\Big)\wedge \exp\Big({-\sum_{i=1}^3 d\check\theta_i\wedge d\theta_i}\Big)\bigg)\Bigg) \\
%&=\mathcal{P}^{-1}\Big(\check\pi_\ast(d\theta_1\wedge d\theta_2 \wedge d\theta_3\wedge\exp{\sum_{i=1}^3d\check\theta_i\wedge \eta_i})\Big)\\
%&=\mathcal{P}^{-1}\exp{\sum_{i=1}^3d\check\theta_i\wedge \eta_i}\\
%&=e^{2\check\omega}\,.\\
%\end{split}
%\end{equation*}
%%The lemma now follows from a straightfroward computation.
%\end{proof}

\section{Affine 3-dimensional solvmanifolds}
\label{aff3}
Up to isomorphisms there are only four simply connected unimodular solvable Lie groups of dimension 3 (see \cite{Ausflows}):

\begin{itemize}
\item The abelian Lie group $(\R^3,+)$;

\item The 3-dimensional real Heisenberg group $\mathcal H_3(\R)$, that is the group of upper uni-triangular 3-by-3 real matrices.  

\item ${\rm Sol}_3=E(1,1)$: The group of rigid motions of the Minkowski plane.
Explicitly it is $\R^3$ with the product
$$(x,y,z)\star(x',y',z')=(x+e^zx',y+e^{-z}y',z+z')\,.$$
The group may also be seen as $\R \ltimes_\mu \R^2$ where $\mu(z)(x',y')=(e^zx',e^{-z}y')$.
This is {\em completely solvable}.

A matrix representation is the following 
%$$
%\begin{bmatrix}
%e^{\lambda z} & 0 & 0 & x\\
%0 & e^{-\lambda z} & 0 & y\\
%0 & 0 & 1 & z \\
%0 & 0 & 0 & 1
%\end{bmatrix}
%$$
%with $\lambda$ such that $e^\lambda+e^{-\lambda}$ is a natural number greater than 2. We can take $\lambda = \log \frac{3+\sqrt{5}}{2}$.\\
$$
\begin{pmatrix}
e^{z} & 0 & 0 & x\\
0 & e^{-z} & 0 & y\\
0 & 0 & 1 & z \\
0 & 0 & 0 & 1
\end{pmatrix}\,.
$$
\item  $\widetilde{E(2)}$: The universal cover of the group of rigid motions of the Euclidean plane.
Explicitly it is $\R^3$ with the product
$$(x,y,z)\star(x',y',z')=(x+x'\cos z-y'\sin z,y+y'\cos z+x'\sin z,z+z')\,.$$
The group may also be seen as $\R \ltimes_\mu \R^2$ where $\mu(z)(x',y')=(x'\cos z-y'\sin z,y'\cos z+x'\sin z)$. This is {\em non completely solvable}.
\end{itemize}

The lattices of such solvable groups are classified up to conjugacy in \cite{Ausflows}.

Complete left-invariant affine structures on  3-dimensional simply connected unimodular solvable Lie groups are classified in \cite{FG}.

Now we present the complete left-invariant affine structures giving rise to left-invariant type IIA completely solvable solvmanifolds.

\subsection{$(\R^3,+)$}
The standard trivial affine structure of $\R^3$, together with the standard lattice $\Z^3 \subset \R^3$ gives rise, via the construction of section \ref{construction}, to the trivial flat SU$(3)$- structure on the 6-dimensional torus $T^6$.
However it is possible to twist the affine structure of $\R^3$ to get a non-trivial compact type IIA manifold.

%If we take as developing map  the  ``identity map", as above in the first choice for both groups, we will obtain a trivial affine representation for the group $G=\mathbb{R}^3$. This would imply trivial linear representation and therefore trivial monodromy. Consequently, this would lead to a six-dimensional example isomorphic to a six-torus and the fibration being trivial. We will then exclude this from our analysis.
Consider the affine structure $\mathcal{A}_{(\mathbb{R}^3,\bowtie)}$ given by the following developing map:
%\subsubsection{Twisted developing map for $(\mathbb{R}^3,+)$}
%
%Take $\mathbb{R}^3$ with coordinates $(x_1,x_3,x_5)$ and choose as developing map

\begin{equation}
	\text{D}:\begin{pmatrix}
		x_1\\
		x_2\\
		x_3\\
	\end{pmatrix}\longmapsto \begin{pmatrix}
		x_1\\
		x_2\\
		x_3+x_1x_2
	\end{pmatrix}
\end{equation}
%with inverse
%\begin{equation}
%	\text{D}^{-1}:\begin{pmatrix}
%		v_1\\
%		v_2\\
%		v_3\\
%	\end{pmatrix}\longmapsto \begin{pmatrix}
%		v_1\\
%		v_2\\
%		v_3-v_1v_2
%	\end{pmatrix}
%\end{equation}

The corresponding representation $\alpha: (\R^3,+) \to {\rm Aff}(\R^3)$ is given by

\begin{equation}
\begin{split}
\alpha \begin{pmatrix}
		x_1\\
		x_2\\
		x_3\\
	\end{pmatrix} 
	\begin{pmatrix}
		v_1\\
		v_2\\
		v_3\\
	\end{pmatrix}=
	\begin{pmatrix}
              1 & 0 & 0 \\
              0 & 1 & 0\\
              x_2& x_1 & 1 \\
          \end{pmatrix}\begin{pmatrix}
          v_1\\
          v_2\\
          v_3\\
      \end{pmatrix}+\begin{pmatrix}
      x_1\\x_2\\x_3+x_1x_2\\
  \end{pmatrix}
\end{split}
\end{equation}

Choosing again the standard lattice $\Z^3 \subset \R^3$ we get the following affine holonomy of $T^3 = \R^3/\Z^3$:

\begin{equation}
\begin{split}
\alpha \begin{pmatrix}
		n_1\\
		n_2\\
		n_3\\
	\end{pmatrix} 
	\begin{pmatrix}
		v_1\\
		v_2\\
		v_3\\
	\end{pmatrix}=
	\begin{pmatrix}
              1 & 0 & 0 \\
              0 & 1 & 0\\
              n_2& n_1 & 1 \\
          \end{pmatrix}\begin{pmatrix}
          v_1\\
          v_2\\
          v_3\\
      \end{pmatrix}+\begin{pmatrix}
      n_1\\n_2\\n_3+n_1n_2\\
  \end{pmatrix}
\end{split}
\end{equation}

In the next section we will prove that this affine structure gives rise to the type IIA nilmanifold corresponding to the algebra $(0,0,0,0,e^{12},e^{13})$.

\subsection{$\mathcal H_3(\mathbb{R})$}
Consider first the developing map 

\begin{equation}
	\begin{split}
		 \text{D}:\mathcal H_3(\mathbb{R}) & \longrightarrow\mathbb{R}^3 \\
		\begin{pmatrix}
			1 & x_1 & x_3 \\
			0 & 1 & x_2 \\
			0 & 0 & 1 \\
		\end{pmatrix}
		&\longmapsto \begin{pmatrix}
			x_1\\
			x_2\\
			x_3\\
		\end{pmatrix}\\
	\end{split}    
\end{equation}

For $g=\begin{pmatrix}
	1 & x_1 & x_3\\
	0 & 1 & x_2\\
	0 & 0 & 1 \\
\end{pmatrix}$ and $v=\begin{pmatrix}v_1\\v_2\\v_3\\\end{pmatrix}\in\mathbb{R}^3$ we compute $\alpha=\text{D}\circ L_g \circ \text{D}^{-1}$:

\begin{equation}\label{rep}
	\begin{split}
		\alpha(g)(v) & = 
%		\text{Dev}\circ L_g \circ \text{Dev}^{-1}(v) \\
%		          & = \text{Dev}\circ L_g \Bigg(\begin{pmatrix}
%		          	1 & v_1 & v_3\\
%		          	0 & 0 & v_2\\
%		          	0 & 0 & 1\\
%		          \end{pmatrix}\Bigg)\\
%	             & = \text{Dev}\Bigg(\begin{pmatrix}
%	             	1 & x_1 & x_3\\
%	             	0 & 0 & x_2\\
%	             	0 & 0 & 1\\
%	             \end{pmatrix}\begin{pmatrix}
%	             1 & v_1 & v_3\\
%	             0 & 1 & v_2\\
%	             0 & 0 & 1\\
%             \end{pmatrix}\Bigg)\\
%                &=\text{Dev}\Bigg(\begin{pmatrix}
%                	1 & x_1+v_1 & x_3+v_3+x_1v_2\\
%                	0 & 1 & x_2+v_2\\
%                	0 & 0 & 1\\
%                \end{pmatrix}\Bigg)\\
%              &=\begin{pmatrix}
%              	x_1+v_1\\
%              	x_2+v_2\\
%              	x_3+v_3+x_1v_2\\
%              \end{pmatrix}=
\begin{pmatrix}
          	1 & 0 & 0\\
          	0 & 1 & 0\\
          	0 & x_1 & 1\\
          \end{pmatrix}\begin{pmatrix}
          v_1\\
          v_2\\
          v_3\\
      \end{pmatrix}+\begin{pmatrix}
      x_1\\
      x_2\\
      x_3\\
  \end{pmatrix}\,.\\
	\end{split}
\end{equation}

Choosing the standard lattice $\mathcal H_3(\Z) \subset \mathcal H_3(\R)$ of matrices with integral entries we get the following affine holonomy of the Heisenberg manifold $\mathcal H_3(\R) / \mathcal H_3(\Z)$:

\begin{equation}
\begin{split}
\alpha \begin{pmatrix}
		 1 & n_1 & n_3 \\
              0 & 1 & n_2\\
              0 & 0 & 1 \\
	\end{pmatrix} 
	\begin{pmatrix}
		v_1\\
		v_2\\
		v_3\\
	\end{pmatrix}=
	\begin{pmatrix}
              1 & 0 & 0 \\
              0 & 1 & 0\\
              0 & n_1 & 1 \\
          \end{pmatrix}\begin{pmatrix}
          v_1\\
          v_2\\
          v_3\\
      \end{pmatrix}+\begin{pmatrix}
      n_1\\n_2\\n_3\\
  \end{pmatrix}\,.
\end{split}
\end{equation}

In the next section we will prove that this affine structure, denoted with $\mathcal{A}_{(\mathcal{H}_3(\mathbb{R}),0)}$, gives rise to the type IIA nilmanifold corresponding again to the algebra $(0,0,0,0,e^{12},e^{13})$, but with a choice of Lagrangian fibration different from the one obtained from the twisted affine structure on the torus. 
%\begin{equation}
%	\alpha(\gamma)(v)=\begin{pmatrix} 1 & 0 & 0 \\ 0 & 1 & 0 \\ 0 & n_1 & 1\\ \end{pmatrix}\begin{pmatrix}
%		v_1\\ v_2 \\ v_3\\
%	\end{pmatrix}+\begin{pmatrix}
%	n_1\\ n_2 \\ n_3
%\end{pmatrix}
%\end{equation}
%
%so that the assignment
%
%\begin{equation}
%	\begin{pmatrix}
%		1 & n_1 & n_3\\
%		0 &  1 & n_2\\
%		0 & 0 & 1\\
%	\end{pmatrix}\longmapsto \Bigg(\begin{pmatrix}
%	1 & 0 & 0\\
%	0 & 1 & 0\\
%	0 & n_1 & 1\\
%\end{pmatrix},\begin{pmatrix}
%n_1\\
%n_2\\
%n_3\\
%\end{pmatrix}\Bigg)
%\end{equation}
%
%is a well-defined homomorphism from $H_3(\mathbb{Z})$ to Aff$_\mathbb{Z}(\mathbb{R}^3)=\text{GL}(3,\mathbb{Z})\ltimes\mathbb{R}^3$ (the translation part is $\mathbb{Z}^3$ indeed, but we are interested just in the linear part).

%Set now for future reference
%
%
%\begin{equation}
%\begin{split}
%	&\lambda_{N,1}:=\text{Lin}\circ\alpha: g\longmapsto\begin{pmatrix}
%		1 & 0 & 0\\
%		0 & 1 & 0\\
%		0 & x_1 & 1\\
%	\end{pmatrix}\\
%   &\mathfrak{l}_{N,1}:=\text{Lin}\circ\mathfrak{a}:\gamma\longmapsto\begin{pmatrix}
%	1 & 0 & 0\\
%	0 & 1 & 0\\
%	0 & n_1 & 1\\
%\end{pmatrix}\\
%\end{split}
%\end{equation}

%\subsubsection{Twisted developing map for $H_3(\mathbb{R})$}

Consider now the following family of developing maps parametrised by $\lambda \in \R \setminus \{0,1\}$: 

\begin{equation}
	\begin{split}
		& D:\mathcal H_3(\mathbb{R})\longrightarrow\mathbb{R}^3 \\
		&\begin{pmatrix}
			1 & x_1 & x_3 \\
			0 & 1 & x_2 \\
			0 & 0 & 1 \\
		\end{pmatrix}\longmapsto \begin{pmatrix}
			x_1\\
			\lambda x_2\\
			(\lambda-1)x_3+x_1x_2\\
		\end{pmatrix}\\
	\end{split}    
\end{equation}  

%whose inverse is $\text{D}^{-1}:\mathbb{R}^3\rightarrow \mathcal H_3(\mathbb{R})$
%
%$$\begin{pmatrix}
%	v_1\\
%	v_2\\
%	v_3\\
%\end{pmatrix}\longmapsto\begin{pmatrix}
%	1 & v_1 & \frac{1}{\lambda-1}(v_3-\frac{v_1v_2}{\lambda}) \\
%	0 & 1 & \frac{v_2}{\lambda}\\
%	0 & 0 & 1\\
%\end{pmatrix}\,.$$
%where $\lambda\in\mathbb{R}\backslash \{0,1\}$.

In this case for the affine representation $\alpha$ we get:

\begin{equation}
	%\begin{split}
		\alpha(g)(v) 
%		\text{D}\circ L_g \circ \text{D}^{-1}(v) 
%		& = \text{Dev}\circ L_g \Bigg(\begin{pmatrix}
%			1 & v_1 & \frac{1}{\lambda-1}(v_3-\frac{v_1v_2}{\lambda}) \\
%			0 & 1 & \frac{v_2}{\lambda}\\
%			0 & 0 & 1\\
%		\end{pmatrix}\Bigg)\\
%		& = \text{Dev}\Bigg(\begin{pmatrix}
%			1 & x_1 & x_3\\
%			0 & 0 & x_2\\
%			0 & 0 & 1\\
%		\end{pmatrix}\begin{pmatrix}
%			1 & v_1 & \frac{1}{\lambda-1}(v_3-\frac{v_1v_2}{\lambda}) \\
%			0 & 1 & \frac{v_2}{\lambda}\\
%			0 & 0 & 1\\
%		\end{pmatrix}\Bigg)\\
%		&=\text{Dev}\Bigg(\begin{pmatrix}
%			1 & x_1+v_1 & \frac{1}{\lambda-1}(v_3-\frac{v_1v_2}{\lambda})+x_3+\frac{x_1v_2}{\lambda} \\
%			0 & 1 & x_2+\frac{v_2}{\lambda}\\
%			0 & 0 & 1\\
%		\end{pmatrix}\Bigg)\\
%		&=\begin{pmatrix}
%			x_1+v_1\\
%			\lambda x_2+v_2\\
%			(\lambda-1)x_3+v_3-\frac{v_1v_2}{\lambda}+\frac{\lambda-1}{\lambda}x_1v_2+(x_1+v_1)(x_2+\frac{v_2}{\lambda})\\
%		\end{pmatrix}\\
%		&=\begin{pmatrix}
%			x_1+v_1\\
%			\lambda x_2+v_2\\
%			(\lambda-1)x_3+v_3+x_1v_2+x_2v_1+x_1x_2\\
%		\end{pmatrix}\\
		=\begin{pmatrix}
			1 & 0 & 0\\
			0 & 1 & 0\\
			x_2 & x_1 & 1\\
		\end{pmatrix}
		\begin{pmatrix}
			v_1\\
			v_2\\
			v_3\\
		\end{pmatrix}+\begin{pmatrix}
			x_1\\
			\lambda x_2\\
			(\lambda-1)x_3+x_1x_2\\
		\end{pmatrix}
	%\end{split}
\end{equation}

Choosing again the standard lattice $\mathcal H_3(\Z) \subset\mathcal H_3(\R)$  we get the following affine holonomy of the Heisenberg manifold $\mathcal H_3(\R) / \mathcal H_3(\Z)$:

\begin{equation}
\begin{split}
\alpha \begin{pmatrix}
		 1 & n_1 & n_3 \\
              0 & 1 & n_2\\
              0 & 0 & 1 \\
	\end{pmatrix} 
	\begin{pmatrix}
		v_1\\
		v_2\\
		v_3\\
	\end{pmatrix}=
	\begin{pmatrix}
              1 & 0 & 0 \\
              0 & 1 & 0\\
              n_2& n_1 & 1 \\
          \end{pmatrix}\begin{pmatrix}
          v_1\\
          v_2\\
          v_3\\
      \end{pmatrix}+\begin{pmatrix}
      n_1\\\lambda n_2\\(\lambda-1)n_3+n_1n_2\\
  \end{pmatrix}
\end{split}
\end{equation}

In the next section we will prove that this family of affine structures, denoted with $\mathcal{A}_{(\mathcal{H}_3(\mathbb{R}),\lambda)}$,  gives rise to three inequivalent type IIA nilmanifolds, all of them with underlying algebra $(0,0,0,e^{12},e^{13},e^{23})$.

%we obtain the following linear representations for $H_3(\mathbb{R})$ and $H_3(\mathbb{Z})$ respectively
%
%
%
%\begin{equation}	
%		\lambda_{N,2}:g\longmapsto\begin{pmatrix}
%			1 & 0 & 0\\
%			0 & 1 & 0\\
%			x_2 & x_1 & 1\\
%		\end{pmatrix} \ \ \ \text{and} \ \ \ 
%		\mathfrak{l}_{N,2}:\gamma\longmapsto\begin{pmatrix}
%			1 & 0 & 0\\
%			0 & 1 & 0\\
%			n_2 & n_1 & 1\\
%		\end{pmatrix}	
%\end{equation}
%
%
%and affine coordinates
%
%
%\begin{equation}
%	\begin{cases}
%		r_1=x_1\\
%		r_2=\lambda x_2\\
%		r_3=(\lambda-1)x_3+x_1x_2
%	\end{cases}
%\end{equation}
%

%And analogously for $\gamma\in\mathbb{Z}^3$. Therefore the linear representation are
%
%
%\begin{equation}
%	\begin{split}
%		&\lambda_\mathbb{T}:=\text{Lin}\circ\alpha:g\longmapsto\begin{pmatrix}
%			1 & 0 & 0 \\
%			0 & 1 & 0\\
%			x_3 & x_1 & 1\\
%		\end{pmatrix}\\
%    & \mathfrak
%    {l}_\mathbb{T}:=\text{Lin}\circ\alpha:\gamma\longmapsto\begin{pmatrix}
%    	1 & 0 & 0 \\
%    	0 & 1 & 0\\
%    	n_3 & n_1 & 1\\
%    \end{pmatrix}\\
%\end{split}
%\end{equation}
%
%and affine coordinates defined as:
%
%\begin{equation}
%	\begin{cases}
%		r_1=x_1\\
%		r_2=x_2\\
%		r_3=x_3+x_1x_2
%	\end{cases}
%\end{equation}

\subsection{$E(1,1)$}

Choose as developing map 
\begin{equation}
	\begin{split}
		& D:E(1,1)\longrightarrow\mathbb{R}^3 \\
		&\begin{pmatrix}
			e^{x_1} & 0 & 0 & x_2\\
			0 & e^{-x_1} & 0 & x_3 \\
			0 & 0 & 1 & x_1  \\
			0 & 0 & 0 & 1\\
		\end{pmatrix}\longmapsto \begin{pmatrix}
			x_1\\
			x_2\\
			x_3\\
		\end{pmatrix}\\
	\end{split}    
\end{equation}

%with inverse $\text{Dev}^{-1}:\mathbb{R}^3\rightarrow E(1,1)$
%
%$$\begin{pmatrix}
%	v_1\\
%	v_2\\
%	v_3\\
%\end{pmatrix}\longmapsto\begin{pmatrix}
%e^{v_1} & 0 & 0 & v_2\\
%0 & e^{-v_1} & 0 & v_3 \\
%0 & 0 & 1 & v_1  \\
%0 & 0 & 0 & 1\\
%\end{pmatrix}$$

Thus for $g=\begin{pmatrix}
	e^{x_1} & 0 & 0 & x_2\\
	0 & e^{-x_1} & 0 & x_3 \\
	0 & 0 & 1 & x_1  \\
	0 & 0 & 0 & 1\\
\end{pmatrix}$ and $v=\begin{pmatrix}v_1\\v_2\\v_3\\\end{pmatrix}\in\mathbb{R}^3$ we have

\begin{equation}\label{repe}
	\begin{split}
		\alpha(g)(v) & = 
%		\text{Dev}\circ L_g \circ \text{Dev}^{-1}(v) \\
%		& = \text{Dev}\circ L_g \Bigg(\begin{pmatrix}
%			e^{v_1} & 0 & 0 & v_2\\
%			0 & e^{-v_1} & 0 & v_3 \\
%			0 & 0 & 1 & v_1  \\
%			0 & 0 & 0 & 1\\
%		\end{pmatrix}\Bigg)\\
%		& = \text{Dev}\Bigg(\begin{pmatrix}
%				e^{x_1} & 0 & 0 & x_2\\
%			0 & e^{-x_1} & 0 & x_3 \\
%			0 & 0 & 1 & x_1  \\
%			0 & 0 & 0 & 1\\
%		\end{pmatrix}\begin{pmatrix}
%		e^{v_1} & 0 & 0 & v_2\\
%		0 & e^{-v_1} & 0 & v_3 \\
%		0 & 0 & 1 & v_1  \\
%		0 & 0 & 0 & 1\\
%	\end{pmatrix}\Bigg)\\
%		&=\text{Dev}\Bigg(\begin{pmatrix}
%				e^{x_1+v_1} & 0 & 0 & x_2+e^{x_1}v_2\\
%			0 & e^{-x_1-v_1} & 0 & x_3+e^{-x_1}v_3 \\
%			0 & 0 & 1 & x_1+v_1  \\
%			0 & 0 & 0 & 1\\
%		\end{pmatrix}\Bigg)\\
%		&=\begin{pmatrix}
%			x_1+v_1\\
%			x_2+e^{x_1}v_2\\
%			x_3+e^{-x_1}v_3\\
%		\end{pmatrix}=
		\begin{pmatrix}
			1 & 0 & 0\\
			0 & e^{x_1} & 0\\
			0 & 0 & e^{-x_1}\\
		\end{pmatrix}\begin{pmatrix}
			v_1\\
			v_2\\
			v_3\\
		\end{pmatrix}+\begin{pmatrix}
			x_1\\
			x_2\\
			x_3\\
		\end{pmatrix}\\
	\end{split}
\end{equation}

%Thus we obtained a homomorphism from $E(1,1)$ to Aff$(\mathbb{R}^3)=\text{GL}(3,\mathbb{R})\ltimes\mathbb{R}^3$
%\begin{equation}
%	\begin{pmatrix}
%			e^{x_1} & 0 & 0 & x_2\\
%		0 & e^{-x_1} & 0 & x_3 \\
%		0 & 0 & 1 & x_1  \\
%		0 & 0 & 0 & 1\\
%	\end{pmatrix}\longmapsto \Bigg( \begin{pmatrix}
%		1 & 0 & 0 \\
%		0 & e^{x_1} & 0 \\
%		0 & 0 & e^{-x_1} \\
%	\end{pmatrix}, \begin{pmatrix}
%		x_1\\
%		x_2\\
%		x_3\\
%	\end{pmatrix}\Bigg)
%\end{equation}

Let $t$ be a real number such that $e^t + e^{-t}$ is an integer bigger than 2.
We call $\Gamma_t$ the subgroup of $E(1,1)$ made by the elements of the form 
%Take now an element $\gamma$ in $\Gamma_t$. It is of the form

$$\gamma=\begin{pmatrix}
	e^{tn_1} & 0 & 0 & n_2+e^{t}n_3\\
	0 &e^{-tn_1}& 0 & n_2+e^{-t}n_3\\
	0 & 0 & 1 & tn_1\\
	0 & 0 & 0 & 1
\end{pmatrix}\,,$$

with $n_1,n_2,n_3\in\mathbb{Z}$. 
It is easy to verify that $\Gamma_t$ is a lattice of $E(1,1)$.
 If we compute again the integral affine representation $\alpha(\gamma)(v)$ we get 

$$\begin{pmatrix}
	1 & 0 & 0\\
	0 &e^{tn_1}& 0 \\
	0 & 0 & e^{-tn_1}
\end{pmatrix}$$
as linear part which does not lie in GL$(3,\mathbb{Z})$. Nevertheless it is conjugated to an element of GL$(3,\mathbb{Z})$ as the following identity shows

\begin{equation}
	\begin{pmatrix}
		1 & 0 & 0 \\
		0 & e^{t} & 0 \\
		0 & 0 & e^{-t} \\
	\end{pmatrix}=  \begin{pmatrix}
		1 & 0 & 0 \\
		0 & 1 & e^t \\
		0 & 1 & e^{-t} \\
	\end{pmatrix} \begin{pmatrix}
		1 & 0 & 0 \\
		0 & 0 & -1 \\
		0 & 1 & e^t + e^{-t} \\
	\end{pmatrix} \begin{pmatrix}
		1 & 0 & 0 \\
		0 & 1 & e^t \\
		0 & 1 & e^{-t} \\
	\end{pmatrix}^{-1}\,.
\end{equation}

%So that
%
%\begin{equation}\label{intmat}
%	\begin{pmatrix}
%		1 & 0 & 0 \\
%		0 & e^{n_1t} & 0 \\
%		0 & 0 & e^{-n_1t} \\
%	\end{pmatrix}=  \begin{pmatrix}
%		1 & 0 & 0 \\
%		0 & 1 & e^t \\
%		0 & 1 & e^{-t} \\
%	\end{pmatrix} \begin{pmatrix}
%		1 & 0 & 0 \\
%		0 & 0 & -1 \\
%		0 & 1 & 3 \\
%	\end{pmatrix}^{n_1} \begin{pmatrix}
%		1 & 0 & 0 \\
%		0 & 1 & e^t \\
%		0 & 1 & e^{-t} \\
%	\end{pmatrix}^{-1}
%\end{equation}

Therefore, though the linear part has not integer entries, it represents an automorphism of the lattice $$\Xi_t=\Big\langle \begin{pmatrix}
	1 \\ 0 \\ 0 \\
\end{pmatrix},\begin{pmatrix}
0 \\ 1 \\ 1\\
\end{pmatrix}, \begin{pmatrix}
0 \\ e^t \\ e^{-t}\\
\end{pmatrix} \Big\rangle_{\mathbb{Z}}$$ inside $\mathbb{R}^3$. 
%We will denote this lattice by $\Xi_t$.
This is exactly the situation described in Remark \ref{remGamma}.

%Finally set
%
%
%%\begin{comment}
%%		\begin{equation}
%%		\begin{split}
%%			&\lambda_{S,1}:=\text{Lin}\circ\alpha:g\longmapsto\begin{pmatrix}
%%				1 & 0 & 0 \\
%%				0 & e^{x_1} & 0 \\
%%				0 & 0 & e^{-x_1} \\
%%			\end{pmatrix}\\
%%			&\mathfrak{l}_{S,1}:=\text{Lin}\circ\mathfrak{a}:\gamma\longmapsto\begin{pmatrix}
%%				1 & 0 & 0 \\
%%				0 & 0 & -1 \\
%%				0 & 1 & 3 \\
%%			\end{pmatrix}^{n_1}\\
%%		\end{split}
%%	\end{equation}
%%\end{comment}
%
%\begin{equation}
%	\begin{split}
%		&\lambda_{S,1}:=\text{Lin}\circ\alpha:g\longmapsto\begin{pmatrix}
%			1 & 0 & 0 \\
%			0 & e^{x_1} & 0 \\
%			0 & 0 & e^{-x_1} \\
%		\end{pmatrix}\\
%		&\mathfrak{l}_{S,1}:=\text{Lin}\circ\mathfrak{a}:\gamma\longmapsto\begin{pmatrix}
%			1 & 0 & 0 \\
%			0 & e^{tn_1} & 0 \\
%			0 & 0 & e^{-tn_1} \\
%		\end{pmatrix}\\
%	\end{split}
%\end{equation}
%	
%
%
%while affine coordinates are then defined as:
%
%\begin{equation}
%	\begin{cases}
%		r_1=x_1\\
%		r_2=x_2\\
%		r_3=x_3
%	\end{cases}
%\end{equation}

In the next section we will prove that this affine structure, denoted with $\mathcal{A}_{(E(1,1),0)}$, gives rise to the type IIA solvmanifold corresponding to the algebra $(15,-25,-35,45,0,0)$ (see \cite{FMOU}).

\subsubsection{Twisted developing map for $E(1,1)$}

Take now as developing map

\begin{equation}
	\begin{split}
		& D:E(1,1)\longrightarrow\mathbb{R}^3 \\
		&\begin{pmatrix}
			e^{x_1} & 0 & 0 & x_2\\
			0 & e^{-x_1} & 0 & x_3 \\
			0 & 0 & 1 & x_1  \\
			0 & 0 & 0 & 1\\
		\end{pmatrix}\longmapsto \begin{pmatrix}
			x_1+x_2x_3\\
			x_2\\
			x_3\\
		\end{pmatrix}\\
	\end{split}    
\end{equation}
%with inverse $\text{Dev}^{-1}:\mathbb{R}^3\rightarrow E(1,1)$
%
%$$\begin{pmatrix}
%	v_1\\
%	v_2\\
%	v_3\\
%\end{pmatrix}\longmapsto\begin{pmatrix}
%	e^{v_1-v_2v_3} & 0 & 0 & v_2\\
%	0 & e^{-v_1+v_2v_3} & 0 & v_3 \\
%	0 & 0 & 1 & v_1-v_2v_3  \\
%	0 & 0 & 0 & 1\\
%\end{pmatrix}$$
%
Again we compute
\begin{equation}
	\begin{split}
		\alpha(g)(v)  = 
%		\text{Dev}\circ L_g \circ \text{Dev}^{-1}(v) \\
%		& = \text{Dev}\circ L_g \Bigg(\begin{pmatrix}
%			e^{v_1-v_2v_3} & 0 & 0 & v_2\\
%		0 & e^{-v_1+v_2v_3} & 0 & v_3 \\
%		0 & 0 & 1 & v_1-v_2v_3  \\
%		0 & 0 & 0 & 1\\
%		\end{pmatrix}\Bigg)\\
%		& = \text{Dev}\Bigg(\begin{pmatrix}
%				e^{x_1} & 0 & 0 & x_2\\
%			0 & e^{-x_1} & 0 & x_3 \\
%			0 & 0 & 1 & x_1  \\
%			0 & 0 & 0 & 1\\
%		\end{pmatrix}\begin{pmatrix}
%			e^{v_1-v_2v_3} & 0 & 0 & v_2\\
%		0 & e^{-v_1+v_2v_3} & 0 & v_3 \\
%		0 & 0 & 1 & v_1-v_2v_3  \\
%		0 & 0 & 0 & 1\\
%		\end{pmatrix}\Bigg)\\
%		&=\text{Dev}\Bigg(\begin{pmatrix}
%				e^{x_1+v_1-v_2v_3} & 0 & 0 & x_2+e^{x_1}v_2\\
%			0 & e^{-x_1-v_1+v_2v_3} & 0 & x_3+e^{-x_1}v_3 \\
%			0 & 0 & 1 & x_1+v_1-v_2v_3  \\
%			0 & 0 & 0 & 1\\
%		\end{pmatrix}\Bigg)\\
%		&=\begin{pmatrix}
%			x_1+v_1+(x_2+e^{x_1}v_2)(x_3+e^{-x_1}v_3)\\
%			 x_2+e^{x_1}v_2\\
%			x_3+e^{-x_1}v_3\\
%		\end{pmatrix}\\
%		&=\begin{pmatrix}
%			x_1+v_1+x_2x_3+v_2v_3+x_2e^{-x_1}v_3+x_3e^{x_1}v_2\\
%			x_2+e^{x_1}v_2\\
%			x_3+e^{-x_1}v_3\\
%		\end{pmatrix}\\
		\begin{pmatrix}
			1& x_3e^{x_1} & x_2e^{-x_1}\\
			0 & e^{x_1} & 0\\
			0 & 0 & e^{-x_1}\\
		\end{pmatrix}\begin{pmatrix}
			v_1\\
			v_2\\
			v_3\\
		\end{pmatrix}+\begin{pmatrix}
			x_1+x_2x_3\\
			x_2\\
		    x_3\\
		\end{pmatrix}\\
	\end{split}
\end{equation}

%and we obtain a different affine representation for $E(1,1)$:
%
%\begin{equation}
%	\begin{pmatrix}
%	e^{x_1} & 0 & 0 & x_2\\
%	0 & e^{-x_1} & 0 & x_3 \\
%	0 & 0 & 1 & x_1  \\
%	0 & 0 & 0 & 1\\
%	\end{pmatrix}\longmapsto \Bigg( \begin{pmatrix}
%		1& x_3e^{x_1} & x_2e^{-x_1}\\
%	0 & e^{x_1} & 0\\
%	0 & 0 & e^{-x_1}\\
%	\end{pmatrix}, \begin{pmatrix}
%		x_1+x_2x_3\\
%		x_2\\
%		x_3\\
%	\end{pmatrix}\Bigg)
%\end{equation}

Take $\gamma\in\Gamma_t$. If we compute  $\alpha(\gamma)(v)$ we obtain as linear part

\begin{equation}
\begin{pmatrix}\label{intmat}
	1 & e^{tn_1}(n_2+e^{-t}n_3) & e^{-tn_1}(n_2+e^{t}n_3)\\
	0 &e^{tn_1}& 0 \\
	0 & 0 & e^{-tn_1}
\end{pmatrix}
\end{equation}

which, again, does not lie in GL$(3,\mathbb{Z})$. 
%If we want to show this is still conjugated to a integral matrix as in (\ref{intmat}), the computation is rather more cumbersome. Nevertheless, take as generators for $\Gamma_t$:
Nevertheless the following identities on the generators of $\Gamma_t$

%\begin{equation}
%	\gamma_1:=\begin{pmatrix}
%		e^{t} & 0 & 0 & 0 \\
%		0 & e^{-t} & 0 & 0 \\
%		0 & 0 & 1 & t &\\
%		0 & 0 & 0 & 1
%	\end{pmatrix} \ \ \ , \ \ \ 	\gamma_2:=\begin{pmatrix}
%	1 & 0 & 0 & 1 \\
%	0 & 1 & 0 & 1 \\
%	0 & 0 & 1 & 0 &\\
%	0 & 0 & 0 & 1
%\end{pmatrix} \ \ \ , \ \ \ 	\gamma_3:=\begin{pmatrix}
%1 & 0 & 0 & e^t \\
%0 & 1 & 0 & e^{-t} \\
%0 & 0 & 1 & 0 &\\
%0 & 0 & 0 & 1
%\end{pmatrix}
%\end{equation}
%
%so that
%
%\begin{equation}
%	\rho(\gamma_1)=\begin{pmatrix}
%		1 & 0 & 0\\
%		0 & e^{t} & 0 \\
%		0 & 0 & e^{-t}\\
%	\end{pmatrix} \ \ \ , \ \ \ \rho(\gamma_2)=\begin{pmatrix}
%	1 & 1 & 1\\
%	0 & 1 & 0\\
%	0 & 0 & 1\\
%\end{pmatrix} \ \ \ , \ \ \ \rho(\gamma_3)=\begin{pmatrix}
%1 & e^{t} & e^{-t}\\
%0 & 1 & 0 \\
%0 & 0 & 1\\
%\end{pmatrix}
%\end{equation}
%We observe that

\begin{equation}\label{reprel}
	\begin{split}
		& 	\begin{pmatrix}
			1 & 0 & 0 \\
			0 & e^{t} & 0 \\
			0 & 0 & e^{-t} \\
		\end{pmatrix}=  \begin{pmatrix}
			1 & 0 & 0 \\
			0 & 1 & e^t \\
			0 & 1 & e^{-t} \\
		\end{pmatrix} \begin{pmatrix}
			1 & 0 & 0 \\
			0 & 0 & -1 \\
			0 & 1 & e^t+e^{-t} \\
		\end{pmatrix} \begin{pmatrix}
			1 & 0 & 0 \\
			0 & 1 & e^t \\
			0 & 1 & e^{-t} \\
		\end{pmatrix}^{-1}\\
	& 	\begin{pmatrix}
		1 & 1 & 1 \\
		0 & 1 & 0 \\
		0 & 0 & 1 \\
	\end{pmatrix}=  \begin{pmatrix}
		1 & 0 & 0 \\
		0 & 1 & e^t \\
		0 & 1 & e^{-t} \\
	\end{pmatrix} \begin{pmatrix}
		1 & 2 & e^t+e^{-t} \\
		0 & 1 & 0 \\
		0 & 0 & 1 \\
	\end{pmatrix} \begin{pmatrix}
		1 & 0 & 0 \\
		0 & 1 & e^t \\
		0 & 1 & e^{-t} \\
	\end{pmatrix}^{-1} \\
&	\begin{pmatrix}
	1 & e^t & e^{-t} \\
	0 & 1 & 0 \\
	0 & 0 & 1 \\
\end{pmatrix}=  \begin{pmatrix}
	1 & 0 & 0 \\
	0 & 1 & e^t \\
	0 & 1 & e^{-t} \\
\end{pmatrix} \begin{pmatrix}
	1 & e^t+e^{-t} & e^{2t}+e^{-2t} \\
	0 & 1 & 0 \\
	0 & 0 & 1 \\
\end{pmatrix} \begin{pmatrix}
	1 & 0 & 0 \\
	0 & 1 & e^t \\
	0 & 1 & e^{-t} \\
\end{pmatrix}^{-1}\\
	\end{split}
\end{equation}

show that we can interpret the matrix (\ref{intmat}) as an automorphism of the lattice $\Xi_t$.

%Therefore the linear representations are:
%
%%\begin{comment}
%%	\begin{equation}
%%\lambda_{S,2}:=\text{Lin}\circ\alpha:g\longmapsto\begin{pmatrix}
%%			1& x_3e^{x_1} & x_2e^{-x_1}\\
%%			0 & e^{x_1} & 0\\
%%			0 & 0 & e^{-x_1}\\
%%		\end{pmatrix}\\
%%\end{equation}
%%
%%and $\mathfrak{l}_{S,2}$ as the representation acting on generators as:
%%
%%\begin{equation}
%%\gamma_1\overset{\mathfrak{l}_{S,2}}{\longmapsto}\begin{pmatrix}
%%			1 & 3 & 7\\
%%			0 & 0 & -1\\
%%			0 & 1 & 3\\
%%		\end{pmatrix} \ \ \ , \ \ \ 
%%\gamma_2\overset{\mathfrak{l}_{S,2}}{\longmapsto}\begin{pmatrix}
%%		1 & 2 & 3\\
%%		0 & 1 & 0\\
%%		0 & 0 & 1\\
%%	\end{pmatrix} \ \ \ , \ \ \ 
%%\gamma_3\overset{\mathfrak{l}_{S,2}}{\longmapsto}\begin{pmatrix}
%%	1 & 3 & 7\\
%%	0 & 1 & 0\\
%%	0 & 0 & 1\\
%%\end{pmatrix}
%%\end{equation}
%%\end{comment}
%	\begin{equation}
%		\begin{split}
%			&\lambda_{S,2}:=\text{Lin}\circ\alpha:g\longmapsto\begin{pmatrix}
%				1& x_3e^{x_1} & x_2e^{-x_1}\\
%				0 & e^{x_1} & 0\\
%				0 & 0 & e^{-x_1}\\
%			\end{pmatrix}\\
%			&\mathfrak{l}_{S,2}:=\text{Lin}\circ\mathfrak{a}:\gamma\longmapsto\begin{pmatrix}
%				1 & e^{tn_1}(n_2+e^{-t}n_3) & e^{-tn_1}(n_2+e^{t}n_3)\\
%			0 &e^{tn_1}& 0 \\
%			0 & 0 & e^{-tn_1}
%			\end{pmatrix}^{}\\
%		\end{split}
%	\end{equation}
%
%
%Finally, affine coordinates are defined as:
%
%\begin{equation}
%	\begin{cases}
%		r_1=x_1+x_2x_3\\
%		r_2=x_2\\
%		r_3=x_3
%	\end{cases}
%\end{equation}

In the next section we will prove that this affine structure, denoted with $\mathcal{A}_{(E(1,1),\bowtie)}$, gives rise to the type IIA solvmanifold corresponding to the algebra $(16+35,-25+45,36,-46,0,0)$ (see \cite{FMOU}). 
%
%\begin{rem}
%	It is nice to observe that the integral properties of the previous matrices are linked to the algebraic properties of the roots of the polynomial $x^2-kx+1$ for $x=e^t$.
%\end{rem}

\section{proof of Theorem 1: Semi-flat six-dimensional mirror pairs}

In this section we apply Theorem \ref{main} to build the six-dimensional Lie groups $G(\mathcal{A})=G\ltimes_{\rho^*}(\mathbb{R}^3)^*$ and $\cc G(\mathcal{A})= G\ltimes_\rho\mathbb{R}^3$  where $\mathcal A$ is one of the affine structures presented in section \ref{aff3}. We describe the group law and its Lie (co)algebra. Also we relate the algebras obtained with the ones from the various classifications.

Note that we recover all of the type IIA completely solvable Lie algebras listed in table \ref{algebreIIA} except case 3. The corresponding Lie group (which is isomorphic to $E(1,1) \times E(1,1)$) does not admit a semidirect product decomposition $G \ltimes \R^3$ giving rise to a {\em Lagrangian} fibration with respect to the relevant symplectic structure.

\subsection{Twisted affine structure of $\mathbb{T}^3$}

\subsubsection{$G(\mathcal{A}_{(\mathbb{R}^3,\bowtie)})$ }

The six-dimensional Lie group $G(\mathcal{A}_{(\mathbb{R}^3,\bowtie)})$ associated to the twisted affine structure of the abelian $\R^3$ is $\R^6$ with the multiplication

\[
\begin{split}
(x_1,x_2,x_3  & ,y_1,y_2,y_3)(x'_1,x'_2,x'_3,y'_1,y'_2,y'_3) = \\
	& (x_1+x_1',x_2+x_2',x_3+x_3',y_1+y_1'-x_2y_3',y_2+y_2'-x_1y_3',y_3+y_3')
\end{split}
\]

%The differential of left multiplication is
%\begin{equation}
%	\begin{pmatrix}
%		1 &  0 & 0 & 0 & 0 & 0 \\
%		0 & 1 & 0 & 0 & 0 & 0 \\
%		0 & 0 & 1 & 0 & 0 &0  \\
%		0 & 0 & 0 &  1 & 0 & -x_1 \\
%		0 & 0 & 0 & 0 & 1 & -x_2\\
%		0 & 0 & 0 & 0 & 0 & 1
%	\end{pmatrix}
%\end{equation}

%which gives the following basis of left-invariant vector fields

%\begin{equation}
%	\begin{split}
%		&E_1=\frac{\partial}{\partial x_1} \ \ \ , \ \ \  E_2=\frac{\partial}{\partial x_2} \ \ \ , \ \ \ E_3=\frac{\partial}{\partial x_3}\\
%		& E_4=\frac{\partial}{\partial x_4} \ \ \, \ \ \ E_5=\frac{\partial}{\partial x_5} \ \ \, \ \ \ E_6=\frac{\partial}{\partial x_6}-x_1\frac{\partial}{\partial x_4}-x_2\frac{\partial}{\partial x_5}
%	\end{split}
%\end{equation}

%The only non-trivial brackets are

%\begin{equation}
%	[E_1,E_6]=-E_4 \ \ \, \ \ \ [E_2,E_6]=-E_5
%\end{equation}

which gives the following basis of left-invariant 1-forms

\begin{equation}
	\begin{split}
		& e^1=dy_1+x_2dy_3 \ \ \ , \ \ \  e^2=dy_2+x_1dy_3 \ \ \ , \  \ e^3=dy_3 \\ 
		&  e^4=dx_1 \ \ \ , \ \ \  e^5=dx_2 \ \ \ , \ \ \  e^6=dx_3 
	\end{split}
\end{equation}

with

\begin{equation}
	\begin{split}		
		& de^1=-e^{35} \ \ \ , \ \ \ de^2= -e^{34}\ \ \ , \ \ \ de^3=0\\
		&de^4=0 \ \ \ , \ \ \ de^5=0 \ \ \ , \ \ \ de^6=0
	\end{split}
\end{equation}

The algebra obtained is isomorphic to $(0,0,0,0,12,13)$, see table \ref{algebreIIA}.

The action-angle coordinates are 
\begin{equation}
	\begin{cases}
		r_1=x_1\\
		r_2=x_2\\
		r_3=x_3+x_1x_2\\
	\end{cases} \ \ \ , \ \ \ \begin{cases}
		\theta_1=y_1\\
		\theta_2=y_2\\
		\theta_3=y_3\\
	\end{cases}
\end{equation}

In these coordinates the coframe of left-invariant 1-forms rewrites as

\begin{equation}
	\begin{split}
		& e^1=d\theta_1+r_2d\theta_3 \\
		& e^2=d\theta_2+r_1d\theta_3 \\
		& e^3=d\theta_3\\
		& e^4=dr_1\\
		& e^5=dr_2\\
		& e^6=dr_3-r_2dr_1-r_1dr_2
	\end{split}
\end{equation}

The induced left-invariant symplectic structure is $\omega=e^{14}+e^{25}+e^{36}=\sum_{i=1}^3d\theta_i\wedge dr_i$.

Now consider the distinguished 3-form $$\Omega=(e^1+ie^4)\w (e^2+ie^5)\w(e^3+ie^6)$$
induced by the choice of the developing map.

One easily verifies that $d \,\text{Re} \ \Omega=0$ and this indeed corresponds to case 1 in table \ref{algebreIIA}.

%\begin{equation}
%	\begin{split}
%		&\omega=e^{41}+e^{23}+e^{65}\\
%		&\Omega=(e^4+ie^1)\w (e^2+ie^3)\w(e^6+ie^5)\\
%	\end{split}
%\end{equation}

%Take
%
%
%\begin{equation}
%	\begin{split}
%		& \text{Re} \ \Omega= e^{426}-e^{435}-e^{125}-e^{136}\\
%		& \text{Im} \ \Omega= e^{425}+e^{436}+e^{126}-e^{135}\\
%	\end{split}
%\end{equation}
%
%\begin{equation}
%	\begin{split}
%		& d\omega=e^{136}-e^{136}=0\\
%		& d\text{Re} \ \Omega=0\\
%		& d \text{Im} \ \Omega=-e^{2536}-e^{1456}\neq 0\\
%	\end{split}
%\end{equation}
%
%and
%
%\begin{equation}
%	\frac{1}{8}\Omega\wedge\bar{\Omega}=-i \ e^{123456}=-i\frac{\omega^3}{6}
%\end{equation}
%
%with $F\equiv 8$.

%\subsubsection{TY Diamond for $\mathfrak{g}_{\mathbb{T}}$}
%
%\begin{equation}
%	\begin{tikzcd}
%		& &    & 1 \\
%		&   & 1 &  & 3 \\
%		& 2 & & 6 & & 3 \\
%		1 &  & 4 & & 7 &  & 1 \\
%		&  2 & & 6 & & 3 \\
%		&  & 1 &  & 3 \\
%		& & &1
%	\end{tikzcd}
%\end{equation}

%\subsubsection{IIB Equations on $\cc{\mathfrak{g}}_{\mathbb{T}}=(0,0,0,0,0,12+34)$}

\subsubsection{$\cc G(\mathcal{A}_{(\mathbb{R}^3,\bowtie)})$ }
\label{twistedR3IIB}
The dual six-dimensional Lie group $\cc G(\mathcal{A}_{(\mathbb{R}^3,\bowtie)})$  associated to the twisted affine structure of the abelian $\R^3$ is $\R^6$ with the following multiplication

\[
\begin{split}
(x_1,x_2,x_3  & ,\cc y_1,\cc y_2,\cc y_3)(x'_1,x'_2,x'_3,\cc y'_1,\cc y'_2,\cc y'_3) = \\
	&(x_1+x_1',x_2+x_2',x_3+x_3',\cc y_1+\cc y_1',\cc y_2+\cc y_2',\cc y_3+\cc y_3'+x_2\cc y_1'+x_1\cc y_2')\\
\end{split}
\]

%The differential of left multiplication is
%\begin{equation}
%	\begin{pmatrix}
%		1 &  0 & 0 & 0 & 0 & 0 \\
%		0 & 1 & 0 & 0 & 0 & 0 \\
%		0 & 0 & 1 & 0 & 0 &0  \\
%		0 & 0 & 0 &  1 & 0 & 0 \\
%		0 & 0 & 0 & 0 & 1 & 0\\
%		0 & x_1 & 0 & x_3 & 0 & 1
%	\end{pmatrix}
%\end{equation}

%which gives the following basis of left-invariant vector fields

%begin{equation}
%	\begin{split}
%		&\cc E_1=\frac{\partial}{\partial x_1} \ \ \ , \ \ \  \cc E_2=\frac{\partial}{\partial x_2}+x_1\frac{\partial}{\partial x_6} \ \ \ , \ \ \ \cc E_3=\frac{\partial}{\partial x_3}\\
%		& \cc E_4=\frac{\partial}{\partial x_4}+x_3\frac{\partial }{\partial x_6} \ \ \, \ \ \ \cc E_5=\frac{\partial}{\partial x_5} \ \ \, \ \ \ \cc E_6=\frac{\partial}{\partial x_6}
%	\end{split}
%\end{equation}

%The only non-trivial brackets are

%\begin{equation}
%	[\cc E_1,\cc E_2]=\cc E_6 \ \ \, \ \ \ [\cc E_3,\cc E_4]=\cc E_6
%\end{equation}

which gives the following basis of left-invariant 1-forms

\begin{equation}
	\begin{split}
	 		& \cc e^1=d\cc y_1 \ \ \ , \ \ \ \cc e^2=d\cc y_2 \ \ \ , \ \ \ \cc e^3=d\cc y_3-x_1d\cc y_2-x_2d\cc y_1\\
			&  \cc e^4=dx_1 \ \ \ , \ \ \  \cc e^5=dx_2 \ \ \ , \ \ \  \cc e^6=dx_3 
	\end{split}
\end{equation}
with differentials
\begin{equation}
	\begin{split}
		& d\cc e^1=0 \ \ \ , \ \ \ d\cc e^2=0 \ \ \ , \ \ \ d\cc e^3=\cc e^{24}+\cc e^{15}\\
			&d\cc e^4=0 \ \ \ , \ \ \ d\cc e^5=0 \ \ \ , \ \ \ d\cc e^6=0\\
	\end{split}
\end{equation}

The dual action-angle coordinates are 
\begin{equation}
	\begin{cases}
		r_1=x_1\\
		r_2=x_2\\
		r_3=x_3+x_1x_2\\
	\end{cases} \ \ \ , \ \ \ \begin{cases}
		\cc\theta_1=\cc y_1\\
		\cc\theta_2=\cc y_2\\
		\cc\theta_3=\cc y_3\\
	\end{cases}
\end{equation}

In these coordinates the coframe of left-invariant 1-forms rewrites as

\begin{equation}
	\begin{split}
		& \cc e^1=d\cc\theta_1\\
		& \cc e^2=d\cc\theta_2\\
		& \cc e^3=d\theta_3-r_2d\cc\theta_1-r_1d\cc\theta_2\\
		& \cc e^4=dr_1\\
		& \cc e^5=dr_2\\
		& \cc e^6=dr_3-r_2dr_1-r_1dr_2
	\end{split}
\end{equation}

The induced left-invariant complex structure is induced by $\cc\Omega=\big(\cc e^1+i\cc e^4\big)\w\big(\cc e^2+i\cc e^5\big)\w\big(\cc e^3+i\cc e^6\big)=\bigwedge_{k=1}^3\big(d\cc\theta_k+i dr_k\big)$.

Now consider the distinguished 2-form $$\cc\omega=\cc e^{14}+\cc e^{25}+\cc e^{36}$$

One easily verifies that $d \cc\omega^2=0$. The algebra obtained is listed as $\mathfrak{h}_3=(0,0,0,0,0,12+34)$ in table 1 of \cite{LUV}.

\subsection{Untwisted affine structure of $\mathcal{H}_3(\mathbb{R})$}

\subsubsection{$G(\mathcal{A}_{(\mathcal{H}_3(\mathbb{R}),0)})$ }
\label{untwistedH3IIA}
The six-dimensional Lie group $G(\mathcal{A}_{(\mathcal{H}_3(\mathbb{R}),0)})$ associated to the untwisted affine structure of the Heisenberg group $\mathcal{H}_3(\mathbb{R})$ is $\R^6$ with the following multiplication

\[
\begin{split}
	(x_1,x_2,x_3  & ,y_1,y_2,y_3)(x'_1,x'_2,x'_3,y'_1,y'_2,y'_3) = \\
	& (x_1+x_1',x_2+x_2',x_3+x_3'+x_1x_2',y_1+y_1',y_2+y_2'-x_1y_3',y_3+y_3')
\end{split}
\]

which gives the following basis of left-invariant 1-forms

\begin{equation}
	\begin{split}
		& e^1=dy_1 \ \ \ , \ \ \  e^2=dy_2+x_1dy_3 \ \ \ , \ \ \ e^3=dy_3 \\
			&  e^4=dx_1 \ \ \ , \ \ \  e^5=dx_2 \ \ \ , \ \ \  e^6=dx_3-x_1dx_2 \\ 
	\end{split}
\end{equation}

with

\begin{equation}
	\begin{split}
			& de^1=0 \ \ \ , \ \ \ de^2= -e^{34}\ \ \ , \ \ \ de^3=0\\
		&de^4=0 \ \ \ , \ \ \ de^5=0 \ \ \ , \ \ \ de^6=-e^{45}\,.\\
	\end{split}
\end{equation}

The algebra obtained is isomorphic to $(0,0,0,0,12,13)$, see table \ref{algebreIIA}.

The action-angle coordinates are 
\begin{equation}
	\begin{cases}
		r_1=x_1\\
		r_2=x_2\\
		r_3=x_3\\
	\end{cases} \ \ \ , \ \ \ \begin{cases}
		\theta_1=y_1\\
		\theta_2=y_2\\
		\theta_3=y_3\\
	\end{cases}
\end{equation}

In these coordinates the coframe of left-invariant 1-forms rewrites as

\begin{equation}
	\begin{split}
		& e^1=d\theta_1 \\
		& e^2=d\theta_2+r_1d\theta_3 \\
		& e^3=d\theta_3\\
		& e^4=dr_1\\
		& e^5=dr_2\\
		& e^6=dr_3-r_1dr_2
	\end{split}
\end{equation}

The induced left-invariant symplectic structure is $\omega=e^{14}+e^{25}+e^{36}=\sum_{i=1}^3d\theta_i\wedge dr_i$.

Now consider the distinguished 3-form $$\Omega=(e^1+ie^4)\w (e^2+ie^5)\w(e^3+ie^6)$$
induced by the choice of the developing map.

One easily verifies that $d\, \text{Re} \ \Omega=0$ and this again corresponds to case 1 in table \ref{algebreIIA}. 

As type IIA manifolds $G(\mathcal{A}_{(\mathcal{H}_3(\mathbb{R}),0)})$ and $G(\mathcal{A}_{(\mathbb{R}^3,\bowtie)})$ are equivariantly  isomorphic but the Lagrangian fibrations are different.
In other words the relevant compact six-dimensional type IIA manifold $X$ admits two inequivalent Lagrangian torus fibrations giving rise to two non-isomorphic semi-flat mirror pairs.
This is also reflected in the refined Tseng-Yau cohomology of the two cases.

We will see in the next subsection that the two mirror complex partners are even not diffeomorphic.

\subsubsection{$\cc G(\mathcal{A}_{(\mathcal{H}_3(\mathbb{R}),0)})$ }
\label{untwistedH3IIB}
The dual six-dimensional Lie group $\cc G(\mathcal{A}_{(\mathcal{H}_3(\mathbb{R}),0)})$  associated to the untwisted affine structure of $\mathcal H_3(\R)$ is $\R^6$ with the following multiplication

\[
\begin{split}
	(x_1,x_2,x_3  & ,\cc y_1,\cc y_2,\cc y_3)(x'_1,x'_2,x'_3,\cc y'_1,\cc y'_2,\cc y'_3) = \\
	&(x_1+x_1',x_2+x_2',x_3+x_3'+x_1x_2',\cc y_1+\cc y_1',\cc y_2+\cc y_2',\cc y_3+\cc y_3'+x_1\cc y_2')\\
\end{split}
\]

which gives the following basis of left-invariant 1-forms

\begin{equation}
	\begin{split}
		& \cc e^1=d\cc y_1 \ \ \ , \ \ \ \cc e^2=d\cc y_2 \ \ \ , \ \ \ \cc e^3=d\cc y_3-x_1d\cc y_2\\ 
		&  \cc e^4=dx_1 \ \ \ , \ \ \  \cc e^5=dx_2 \ \ \ , \ \ \  \cc e^6=dx_3-x_1dx_2 \\
	\end{split}
\end{equation}

with

\begin{equation}
	\begin{split}
		& d\cc e^1=0 \ \ \ , \ \ \ d\cc e^2=0 \ \ \ , \ \ \ d\cc e^3=\cc e^{24}\\
		&d\cc e^4=0 \ \ \ , \ \ \ d\cc e^5=0 \ \ \ , \ \ \ d\cc e^6=-\cc e^{45}\,.\\
	\end{split}
\end{equation}

The dual action-angle coordinates are 
\begin{equation}
	\begin{cases}
		r_1=x_1\\
		r_2=x_2\\
		r_3=x_3\\
	\end{cases} \ \ \ , \ \ \ \begin{cases}
		\cc\theta_1=\cc y_1\\
		\cc\theta_2=\cc y_2\\
		\cc\theta_3=\cc y_3\\
	\end{cases}
\end{equation}

In these coordinates the coframe of left-invariant 1-forms rewrites as

\begin{equation}
	\begin{split}
		& \cc e^1=d\cc\theta_1\\
		& \cc e^2=d\cc\theta_2\\
		& \cc e^3=d\cc\theta_3-r_1d\cc\theta_2\\
		& \cc e^4=dr_1\\
		& \cc e^5=dr_2\\
		& \cc e^6=dr_3-r_1dr_2
	\end{split}
\end{equation}

The left-invariant complex structure is induced by $\cc\Omega=\big(\cc e^1+i\cc e^4\big)\w\big(\cc e^2+i\cc e^5\big)\w\big(\cc e^3+i\cc e^6\big)=\bigwedge_{k=1}^3\big(d\cc\theta_k+i dr_k\big)$.

Now consider the distinguished 2-form $$\cc\omega=\cc e^{14}+\cc e^{25}+\cc e^{36}\,.$$

One easily verifies that $d \cc\omega^2=0$.  The algebra obtained is listed as $\mathfrak{h}_6=(0,0,0,0,12,13)$ in table 1 of \cite{LUV}.
\begin{rem}
{\em 	
The mirror pair arising from this affine structure is special in many respects. First it corresponds to the only known six-dimensional example as presented in \cite[section 7]{LTY}. Here we additionally recognize that the total spaces $T^*B/\Lambda^*$ and $TB/\Lambda$ are both diffeomorphic to the nilmanifold $G/\Gamma$ where $G$ is the group of matrices of the form 
\[
\begin{pmatrix}
	1 & x_1 & x_2 & x_3 & 0 & 0\\
	0 & 1   & x_4 & x_5 & 0 & 0\\
	0 & 0   & 1   & 0   & 0 & 0\\
	0 & 0   & 0   & 1   & 0 & 0\\
	0 & 0   & 0   & 0   & 1 & x_6\\
	0 & 0   & 0   & 0   & 0 & 1
\end{pmatrix}
\] 
and $\Gamma$ is the lattice given by the same matrices with integral entries.

Moreover as anticipated in subsection \ref{untwistedH3IIA} $X=G/\Gamma$ provides an example of a type IIA manifold having two inequivalent torus Lagrangian fibrations with two different IIB partners. 
}

\end{rem}

\subsection{Twisted affine structure of $\mathcal{H}_3(\mathbb{R})$}

\subsubsection{$G(\mathcal{A}_{(\mathcal{H}_3(\mathbb{R}),\bowtie,\lambda)})$ }

The six-dimensional Lie group $G(\mathcal{A}_{(\mathcal{H}_3(\mathbb{R}),\bowtie,\lambda)})$ associated to the twisted family of affine structures of the Heisenberg group $\mathcal{H}_3(\mathbb{R})$ is $\R^6$ with the following multiplication

\[
\begin{split}
	(x_1,x_2,x_3  & ,y_1,y_2,y_3)(x'_1,x'_2,x'_3,y'_1,y'_2,y'_3) = \\
	& (x_1+x_1',x_2+x_2',x_3+x_3'+x_1x_2',y_1+y_1'-x_2y_3',y_2+y_2'-x_1y_3',y_3+y_3')
\end{split}
\]

A basis of left-invariant 1-forms is given by

\begin{equation}
	\begin{split}
		& f^1=dy_1+x_2dy_3 \ \ \ , \ \ \  f^2=dy_2+x_1dy_3 \ \ \ , \ \ \ f^3=dy_3   \\
		& f^4=dx_1 \ \ \ , \ \ \  f^5=dx_2 \ \ \ , \ \ \  f^6=dx_3-x_1dx_2
	\end{split}
\end{equation}

with

\begin{equation}
	\begin{split}
		&df^1=-f^{35} \ \ \ , \ \ \ df^2= -f^{34}\ \ \ , \ \ \ df^3=0\\
		&df^4=0 \ \ \ , \ \ \ df^5=0 \ \ \ , \ \ \ df^6=-f^{45}\,. \\
	\end{split}
\end{equation}

The algebra obtained is isomorphic to $(0,0,0,12,13,23)$, case 2 in table  \ref{algebreIIA}.

The action-angle coordinates are 
\begin{equation}
	\begin{cases}
		r_1=x_1\\
		r_2=\lambda x_2\\
		r_3=(\lambda-1)x_3+x_1x_2\\
	\end{cases} \ \ \ , \ \ \ \begin{cases}
		\theta_1=y_1\\
		\theta_2=y_2\\
		\theta_3=y_3\\
	\end{cases}
\end{equation}

In these coordinates the coframe of left-invariant 1-forms rewrites as

\begin{equation}
	\begin{split}
		& f^1=d\theta_1+\frac{r_2}{\lambda}d\theta_3 \\
		& f^2=d\theta_2+r_1d\theta_3 \\
		& f^3=d\theta_3\\
		& f^4=dr_1\\
		& f^5=\frac{dr_2}{\lambda}\\
		& f^6=\frac{1}{\lambda-1}dr_3-\frac{r_2}{\lambda(\lambda-1)}dr_1-\frac{r_1}{\lambda-1}dr_2
	\end{split}
\end{equation}

The induced left-invariant symplectic structure is $\omega_\lambda=\sum_{i=1}^3d\theta_i\wedge dr_i=f^{14}+\lambda f^{25}+ (\lambda -1)f^{36}$.
Note that the frame $e_1,\dots,e_6$ of section \ref{distinguished} is given by
\begin{equation*}
	\begin{split}
		& e^1=f^1 \quad e^2=f^2 \quad e^3=f^3\\
		& e^4=f^4 \quad e^5=\lambda f^5 \quad e^6=(\lambda-1)f^6
	\end{split}
\end{equation*}
%=\sum_{i=1}^3d\theta_i\wedge dr_i$.

One easily checks that the distinguished 3-form $$\Omega_\lambda=(e^1+ie^4)\w (e^2+ie^5)\w(e^3+ie^6)$$
induced by the choice of the developing map has closed real part for every $\lambda \in \R \setminus \{0,1\}$.

%One easily verifies that $d \text{Re} \ \Omega=0$ and 
This type IIA algebra indeed corresponds to case 2 in table \ref{algebreIIA} and appears for the first time in \cite{deBaTomma1}.
According to the value of the parameter $\lambda$ we obtain non-equivalent IIA algebras, see the discussion in remark \ref{remUgarte}.

\subsubsection{$\cc G(\mathcal{A}_{(\mathcal{H}_3(\mathbb{R}),\bowtie,\lambda)})$ }
\label{twistedH3IIB}
The dual six-dimensional Lie group $\cc G(\mathcal{A}_{(\mathcal{H}_3(\mathbb{R}),\bowtie,\lambda)})$  associated to the twisted family of affine structures of the Heisenberg group $\mathcal{H}_3(\mathbb{R})$ is $\R^6$ with the following multiplication

\[
\begin{split}
	(x_1,x_2,x_3  & ,\cc y_1,\cc y_2,\cc y_3)(x'_1,x'_2,x'_3,\cc y'_1,\cc y'_2,\cc y'_3) = \\
	&(x_1+x_1',x_2+x_2',x_3+x_3'+x_1x_2',\cc y_1+\cc y_1',\cc y_2+\cc y_2',\cc y_3+\cc y_3'+x_1\cc y_2'+x_2y_1')\\
\end{split}
\]

A basis of left-invariant 1-forms is given by

\begin{equation}
	\begin{split}
		&  \cc f^1=d\cc y_1 \ \ \ , \ \ \ \cc f^2=d\cc y_2 \ \ \ , \ \ \ \cc f^3=d\cc y_3-x_2d\cc y_1-x_1d\cc y_2  \\
		&\cc f^4=dx_1 \ \ \ , \ \ \  \cc f^5=dx_2 \ \ \ , \ \ \  \cc f^6=dx_3-x_1dx_2\\ 
	\end{split}
\end{equation}

with

\begin{equation}
	\begin{split}
		& d\cc f^1=0 \ \ \ , \ \ \ d\cc f^2=0 \ \ \ , \ \ \ d\cc f^3=\cc f^{24}+\cc f^{15}\\
		&d\cc f^4=0 \ \ \ , \ \ \ d\cc f^5=0 \ \ \ , \ \ \ d\cc f^6=-\cc e^{45}\\
	\end{split}
\end{equation}

The algebra obtained is isomorphic to $\mathfrak{h}_4=(0,0,0,0,12,14+23)$ in \cite{LUV}.

The dual action-angle coordinates are 
\begin{equation}
	\begin{cases}
		r_1=x_1\\
		r_2=\lambda x_2\\
		r_3=(\lambda-1)x_3+x_1x_2\\
	\end{cases} \ \ \ , \ \ \ \begin{cases}
		\cc\theta_1=\cc y_1\\
		\cc\theta_2=\cc y_2\\
		\cc\theta_3=\cc y_3\\
	\end{cases}
\end{equation}

In these coordinates the coframe of left-invariant 1-forms rewrites as

\begin{equation}
	\begin{split}
		& \cc f^1=d\cc\theta_1\\
		& \cc f^2=d\cc\theta_2\\
		& \cc f^3=d\cc\theta_3-\frac{r_2}{\lambda}d\cc\theta_1-r_1d\cc\theta_2\\
		& \cc f^4=dr_1\\
		& \cc f^5=\frac{dr_2}{\lambda}\\
		&\cc f^6=\frac{1}{\lambda-1}dr_3-\frac{r_2}{\lambda(\lambda-1)}dr_1-\frac{r_1}{\lambda-1}dr_2
	\end{split}
\end{equation}

The induced left-invariant complex structure is $\cc\Omega_\lambda=\bigwedge_{k=1}^3\big(d\cc\theta_k+i dr_k\big)=\big(\cc f^1+i\cc f^4\big)\w\big(\cc f^2+i\lambda \cc f^5\big)\w\big(\cc f^3+i (\lambda -1)\cc f^6\big)$.
Again the frame $\cc e_1,\dots,\cc e_6$ of section \ref{distinguished} is given by
\begin{equation}
	\begin{split}
		& \cc e^1=\cc f^1 \quad \cc e^2=\cc f^2 \quad \cc e^3=\cc f^3\\
		& \cc e^4=\cc f^4 \quad \cc e^5=\lambda \cc f^5 \quad \cc e^6=(\lambda-1)\cc f^6
	\end{split}
\end{equation}

Now consider the distinguished 2-form $$\cc\omega_\lambda=\cc e^{14}+\cc e^{25}+\cc e^{36}$$

One easily verifies that $d \cc\omega_\lambda^2=0$.

\begin{rem}
	\label{remUgarte}
	{\em According to the value of the parameter $\lambda$ we obtain non-equivalent IIB algebras. More precisely the Bott-Chern numbers distinguish 3 different cases:
		\begin{itemize} 
			\item $\lambda=-1$.	This type IIB algebra is missing in the classification of \cite{LUV}.
			\item $\lambda=\frac12,2$.
			\item $\lambda\neq -1,\frac12,2$.
		\end{itemize}
	}
\end{rem}

\subsection{Untwisted affine structure of $E(1,1)$}
\subsubsection{$G(\mathcal{A}_{E(1,1),0)})$ }
The six-dimensional Lie group $G(\mathcal{A}_{E(1,1),0)})$ associated to the untwisted affine structure of the group $E(1,1)$ is $\R^6$ with the following multiplication

\[
\begin{split}
	(x_1,x_2,x_3  & ,y_1,y_2,y_3)(x'_1,x'_2,x'_3,y'_1,y'_2,y'_3) = \\
	& (x_1+x_1',x_2+e^{x_1}x_2',x_3+e^{-x_1}x_3',y_1+y_1',y_2+e^{-x_1}y_2',y_3+e^{x_1}y_3')
\end{split}
\]

which gives the following basis of left-invariant 1-forms

\begin{equation}
	\begin{split}
		&  e^1=dy_1 \ \ \ , \ \ \  e^2=e^{x_1}dy_2 \ \ \ , \ \ \ e^3=e^{-x_1}dy_3  \\
		&  e^4=dx_1 \ \ \ , \ \ \  e^5=e^{-x_1}dx_2 \ \ \ , \ \ \  e^6=e^{x_1}dx_3 
	\end{split}
\end{equation}

with

\begin{equation}
	\begin{split}
		& de^1=0 \ \ \ , \ \ \ de^2=-e^{24}\ \ \ , \ \ \ de^3=e^{34}\\
		& de^4=0 \ \ \ , \ \ \ de^5=-e^{45} \ \ \ , \ \ \ de^6=e^{46}
	\end{split}
\end{equation}

The algebra obtained is isomorphic to $(15,-25,-35,45,0,0)$, see table \ref{algebreIIA}.

The action-angle coordinates are 
\begin{equation}
	\begin{cases}
		r_1=x_1\\
		r_2=x_2\\
		r_3=x_3\\
	\end{cases} \ \ \ , \ \ \ \begin{cases}
		\theta_1=y_1\\
		\theta_2=y_2\\
		\theta_3=y_3\\
	\end{cases}
\end{equation}

In these coordinates the coframe of left-invariant 1-forms rewrites as

\begin{equation}
	\begin{split}
		& e^1=d\theta_1 \\
		& e^2=e^{r_1}d\theta_2 \\
		& e^3=e^{-r_1}d\theta_3\\
		& e^4=dr_1\\
		& e^5=e^{-r_1}dr_2\\
		& e^6=e^{r_1}dr_3
	\end{split}
\end{equation}

The induced left-invariant symplectic structure is $\omega=e^{14}+e^{25}+e^{36}=\sum_{i=1}^3d\theta_i\wedge dr_i$.

Now consider the distinguished 3-form $$\Omega=(e^1+ie^4)\w (e^2+ie^5)\w(e^3+ie^6)$$
induced by the choice of the developing map.

One easily verifies that $d \text{Re} \ \Omega=0$ and this indeed corresponds to case 4 in table \ref{algebreIIA}

\subsubsection{$\cc G(\mathcal{A}_{(E(1,1),0)})$ }

The dual six-dimensional Lie group $\cc G(\mathcal{A}_{E(1,1),0)})$  associated to the untwisted affine structure of the completely solvable $E(1,1)$ is $\R^6$ with the following multiplication

\[
\begin{split}
	(x_1,x_2,x_3  & ,\cc y_1,\cc y_2,\cc y_3)(x'_1,x'_2,x'_3,\cc y'_1,\cc y'_2,\cc y'_3) = \\
	&(x_1+x_1',x_2+e^{x_1}x_2',x_3+e^{-x_1}x_3',\cc y_1+\cc y_1',\cc y_2+e^{x_1}\cc y_2',\cc y_3+e^{-x_1}\cc y_3')\\
\end{split}
\]

which gives the following basis of left-invariant 1-forms

\begin{equation}
	\begin{split}
		&  \cc e^1=d\cc y_1 \ \ \ , \ \ \ \cc e^2=e^{-x_1}d\cc y_2 \ \ \ , \ \ \ \cc e^3=e^{x_1}d\cc y_3\\
		& \cc e^4=dx_1 \ \ \ , \ \ \  \cc e^5=e^{-x_1}dx_2 \ \ \ , \ \ \  \cc e^6=e^{x_1}dx_3
	\end{split}
\end{equation}

with

\begin{equation}
	\begin{split}
		&d\cc e^1=0 \ \ \ , \ \ \ d\cc e^2=\cc e^{24} \ \ \ , \ \ \ d\cc e^3= -\cc e^{34}\\
		& d\cc e^4=0 \ \ \ , \ \ \ d\cc e^5=-e^{45} \ \ \ , \ \ \ d\cc e^6=\cc e^{46}\\
	\end{split}
\end{equation}

Note that the algebra obtained is again isomorphic to $(15,-25,-35,45,0,0)$.

The dual action-angle coordinates are 
\begin{equation}
	\begin{cases}
		r_1=x_1\\
		r_2=x_2\\
		r_3=x_3\\
	\end{cases} \ \ \ , \ \ \ \begin{cases}
		\cc\theta_1=\cc y_1\\
		\cc\theta_2=\cc y_2\\
		\cc\theta_3=\cc y_3\\
	\end{cases}
\end{equation}

In these coordinates the coframe of left-invariant 1-forms rewrites as

\begin{equation}
	\begin{split}
		& \cc e^1=d\cc\theta_1\\
		& \cc e^2=e^{-r_1}d\cc\theta_2\\
		& \cc e^3=e^{r_1}d\cc\theta_3\\
		& \cc e^4=dr_1\\
		& \cc e^5=e^{-r_1}dr_2\\
		& \cc e^6=e^{r_1}dr_3
	\end{split}
\end{equation}

The induced left-invariant complex structure is induced by $\cc\Omega=\big(\cc e^1+i\cc e^4\big)\w\big(\cc e^2+i\cc e^5\big)\w\big(\cc e^3+i\cc e^6\big)=\bigwedge_{k=1}^3\big(d\cc\theta_k+i dr_k\big)$.

Now consider the distinguished 2-form $$\cc\omega=\cc e^{14}+\cc e^{25}+\cc e^{36}$$

One easily verifies that $d \cc\omega^2=0$ and this case corresponds to $\mathfrak g_1$ in \cite[Theorem 2.8]{FOU}.

\subsection{Twisted affine structure of $E(1,1)$}

\subsubsection{$G(\mathcal{A}_{(E(1,1),\bowtie)})$ }

The six-dimensional Lie group $G(\mathcal{A}_{(E(1,1),\bowtie)})$ associated to the twisted affine structure of the group $E(1,1)$ is $\R^6$ with the following multiplication

\[
\begin{split}
	(x_1,x_2,x_3  & ,y_1,y_2,y_3)(x'_1,x'_2,x'_3,y'_1,y'_2,y'_3) = \\
	& (x_1+x_1',x_2+e^{x_1}x_2',x_3+e^{-x_1}x_3',y_1+y_1',y_2+e^{-x_1}y_2'-x_3y_1',y_3+e^{x_1}y_3'-x_2y_1')
\end{split}
\]

which gives the following basis of left-invariant 1-forms

\begin{equation}
	\begin{split}
		& e^1=dy_1 \ \ \ , \ \ \  e^2=e^{x_1}dy_2+x_3e^{x_1}dy_1 \ \ \ , \ \ \ e^3=e^{-x_1}dy_3+x_2e^{-x_1}dy_1  \\
		& e^4=dx_1 \ \ \ , \ \ \  e^5=e^{-x_1}dx_2 \ \ \ , \ \ \  e^6=e^{x_1}dx_3 
	\end{split}
\end{equation}

with

\begin{equation}
	\begin{split}
		&de^1=0 \ \ \ , \ \ \ de^2=-e^{24}-e^{16}\ \ \ , \ \ \ de^3=e^{34}-e^{15} \\
		& de^4=0 \ \ \ , \ \ \ de^5=-e^{45} \ \ \ , \ \ \ de^6=e^{46}\\
	\end{split}
\end{equation}

The algebra obtained is isomorphic to $(16+35,-26+45,36,-46,0,0)$, see table \ref{algebreIIA}.

The action-angle coordinates are 
\begin{equation}
	\begin{cases}
		r_1=x_1+x_2x_3\\
		r_2=x_2\\
		r_3=x_3\\
	\end{cases} \ \ \ , \ \ \ \begin{cases}
		\theta_1=y_1\\
		\theta_2=y_2\\
		\theta_3=y_3\\
	\end{cases}
\end{equation}

In these coordinates the coframe of left-invariant 1-forms rewrites as

\begin{equation}
	\begin{split}
		& e^1=d\theta_1 \\
		& e^2=e^{r_1-r_2r_3}d\theta_2+r_3e^{r_3-r_1r_2}d\theta_1 \\
		& e^3=e^{-r_1+r_2r_3}d\theta_3+r_2e^{-r_3+r_1r_2}d\theta_1\\
		& e^4=dr_1-r_2dr_1-r_1dr_2\\
		& e^5=e^{-r_1+r_2r_3}dr_2\\
		& e^6=e^{r_1-r_2r_3}dr_3
	\end{split}
\end{equation}

The induced left-invariant symplectic structure is $\omega=e^{14}+e^{25}+e^{36}=\sum_{i=1}^3d\theta_i\wedge dr_i$.

Now the distinguished 3-form $$\Omega=(e^1+ie^4)\w (e^2+ie^5)\w(e^3+ie^6)$$
induced by the choice of the developing map has closed real part. This type IIA algebra indeed corresponds to case 7 in table \ref{algebreIIA} and appears for the first time in \cite{TommaLuigi}.

\subsubsection{$\cc G(\mathcal{A}_{(E(1,1),\bowtie)})$ }

The dual six-dimensional Lie group $\cc G(\mathcal{A}_{(E(1,1),\bowtie)})$  associated to the twisted affine structure of the group $E(1,1)$ is $\R^6$ with the following multiplication

\[
\begin{split}
	(x_1,x_2,x_3  & ,\cc y_1,\cc y_2,\cc y_3)(x'_1,x'_2,x'_3,\cc y'_1,\cc y'_2,\cc y'_3) = \\
	&(x_1+x_1',x_2+e^{x_1}x_2',x_3+e^{-x_1}x_3',\cc y_1+\cc y_1'+x_3e^{x_1}y_2'+x_2e^{-x_1}y_3',\cc y_2+e^{x_1}\cc y_2',\cc y_3+e^{-x_1}\cc y_3')\\
\end{split}
\]

which gives the following basis of left-invariant 1-forms

\begin{equation}
	\begin{split}
		&  \cc e^1=d\cc y_1-x_3d\cc y_2-x_2d\cc y_3 \ \ \ , \ \ \ \cc e^2=e^{-x_1}d\cc y_2 \ \ \ , \ \ \ \cc e^3=e^{x_1}d\cc y_3 \\
		& \cc e^4=dx_1 \ \ \ , \ \ \  \cc e^5=e^{-x_1}dx_2 \ \ \ , \ \ \  \cc e^6=e^{x_1}dx_3\\ 
	\end{split}
\end{equation}

with

\begin{equation}
	\begin{split}
		& d\cc e^1=-\cc e^{35}-\cc e^{26} \ \ \ , \ \ \ d\cc e^2=\cc e^{24} \ \ \ , \ \ \ d\cc e^3=-\cc e^{34}\\
		&d\cc e^4=0 \ \ \ , \ \ \ d\cc e^5=-\cc e^{45} \ \ \ , \ \ \ d\cc e^6=\cc e^{46}\\
	\end{split}
\end{equation}

%The algebra obtained is again isomorphic to $(24+35,26,36,-46,-56,0)$ in citare FOU.

The dual action-angle coordinates are 
\begin{equation}
	\begin{cases}
		r_1=x_1+x_2x_3\\
		r_2=x_2\\
		r_3=x_3\\
	\end{cases} \ \ \ , \ \ \ \begin{cases}
		\cc\theta_1=\cc y_1\\
		\cc\theta_2=\cc y_2\\
		\cc\theta_3=\cc y_3\\
	\end{cases}
\end{equation}

In these coordinates the coframe of left-invariant 1-forms rewrites as

\begin{equation}
	\begin{split}
		& \cc e^1=d\cc\theta_1-r_3d\cc\theta_2-r_2d\cc\theta_3\\
		& \cc e^2=e^{-r_1+r_2r_3}d\cc\theta_2\\
		& \cc e^3=e^{r_1-r_2r_3}d\cc\theta_3\\
		& \cc e^4=dr_1-r_2dr_1-r_1dr_2\\
		& \cc e^5=e^{-r_1+r_2r_3}dr_2\\
		& \cc e^6=e^{r_1-r_2r_3}dr_3
	\end{split}
\end{equation}

The left-invariant complex structure is induced by $\cc\Omega=\big(\cc e^1+i\cc e^4\big)\w\big(\cc e^2+i\cc e^5\big)\w\big(\cc e^3+i\cc e^6\big)=\bigwedge_{k=1}^3\big(d\cc\theta_k+i dr_k\big)$.

Now consider the distinguished 2-form $$\cc\omega=\cc e^{14}+\cc e^{25}+\cc e^{36}$$

One easily verifies that $d \cc\omega^2=0$ and this case corresponds to $\mathfrak g_5$ in \cite[Theorem 2.8]{FOU}.

\newpage

\section{Table of mirror pairs}

\begin{center}
	\captionof{table}{Mirror Lie algebras \label{mirrortable}}

	\scalemath{0.58}{
	\renewcommand{\arraystretch}{2}	
	%\begin{adjustbox}{angle=270}
		\begin{tabular}{|l|l|l|c|c|c|c|c|c|c|}
			\hline
			Aff. structure &Lie algebra & Mirror Lie algebra & $h_{TY}^{1,0}/h_{BC}^{2,0}$& $h_{TY}^{0,1}/h_{BC}^{3,1}$&$h_{TY}^{2,0}/h_{BC}^{1,0}$&$h_{TY}^{1,1}/h_{BC}^{2,1}$&$h_{TY}^{0,2}/h_{BC}^{3,2}$&$h_{TY}^{2,1}/h_{BC}^{1,1}$&$h_{TY}^{1,2}/h_{BC}^{3,2}$  \\
			\hline
			$\mathcal{A}_{(\mathbb{R}^3,\bowtie)}$&$(0,0,0,0,e^{12},e^{13})$&$(0,0,0,0,0,e^{12}+e^{34})$ & $1$ & $3$ & $2$ & $6$ & $3$ & $4$ & $7$ \\
			\hline
			$\mathcal{A}_{(\mathcal{H}_3(\mathbb{R}),0)}$&$(0,0,0,0,e^{12},e^{13})$& $(0,0,0,0,e^{12},e^{13})$ & $2$ & $2$ & $2$ & $6$ & $3$ & $5$ & $6$ \\
			\hline
			$\mathcal{A}_{(\mathcal{H}_3(\mathbb{R}),\bowtie,\lambda=-1)}$&$(0,0,0,e^{12},e^{13},e^{23})$ & $(0,0,0,0,e^{12},e^{14}+e^{23})$ &$1$ &$2$ &$2$ &$6$ &$3$ &$4$ &$7$ 	 \\
			\hline
			$\mathcal{A}_{(\mathcal{H}_3(\mathbb{R}),\bowtie,\lambda=\frac{1}{2},2)}$&$(0,0,0,e^{12},e^{13},e^{23})$ & $(0,0,0,0,e^{12},e^{14}+e^{23})$ &$1$ &$2$ &$2$ &$6$ &$3$ &$5$ &$6$ 	 \\
			\hline
			$\mathcal{A}_{(\mathcal{H}_3(\mathbb{R}),\bowtie,\lambda\neq -1,\frac{1}{2},2)}$&$(0,0,0,e^{12},e^{13},e^{23})$ & $(0,0,0,0,e^{12},e^{14}+e^{23})$ &$1$ &$2$ &$2$ &$6$ &$3$ &$4$ &$6$ 	 \\
			\hline
			$\mathcal{A}_{(E(1,1),0)}$&$(e^{15},-e^{25},-e^{35},e^{45},0,0)$  & $(e^{15},-e^{25},-e^{35},e^{45},0,0)$ & 	$1$ & $1$ & $1$ & $3$ & $1$ & $3$ & $3$ \\
			\hline
			$\mathcal{A}_{(E(1,1),\bowtie)}$&$(e^{16}+e^{35},- e^{26}+ e^{45},e^{36},-e^{46},0,0)$ & $(e^{24}+e^{35},e^{26},e^{36},-e^{46},-e^{56},0)$ & $1$ & $1$ & $0$ & $2$ & $1$ & $2$ & $1$\\
			\hline
%			$\mathcal{A}_{(E(2),0,\alpha=0)}$&$(e^{25},- e^{15},e^{45},-e^{35},0,0)$ & $(e^{25},- e^{15},e^{45},-e^{35},0,0)$ &  $1$ & $1$ & $1$ & $3$ & $1$ & $3$ & $3$\\
%			\hline
			%$e(1,1)^{\mathbb{C}}$&$(-e^{16}+e^{25},-e^{15}- e^{26},e^{36}-e^{45},e^{35}+ e^{46},0,0)$ & $(-e^{16}+e^{25},-e^{15}- e^{26},e^{36}-e^{45},e^{35}+ e^{46},0,0)$ & $e(1,1)^{\mathbb{C}}$ \\
			%\hline
		\end{tabular}   
}
	%\end{adjustbox}
	\renewcommand{\arraystretch}{1}  

\end{center}

In table \ref{mirrortable} we present all the mirror pairs of solvable Lie algebras constructed in the previous section and the dimension of their (refined) Tseng-Yau and Bott-Chern cohomology groups.
This computation is valid also for all the corresponding compact solvmanifolds except for the complex side of the pair arising from $\mathcal{A}_{(E(1,1),\bowtie)}$, see \cite{AK1,AK2} where also the complete solvability plays a role.

\end{document}